\documentclass[pagesize,pdftex]{scrartcl}
\usepackage{latexsym} 
\usepackage[intlimits]{amsmath}
\usepackage{amsthm}
\usepackage{amsfonts}
\usepackage{amssymb}
\usepackage{amsxtra}
\usepackage{amscd}
\usepackage{ifthen}
\usepackage{graphicx}
\usepackage[shortlabels,inline]{enumitem}
\usepackage{mathrsfs}
\usepackage[pagebackref=true]{hyperref}

\hypersetup{
pdfauthor={Mads Kyed},
pdftitle={Existence and regularity of time-periodic solutions to the three-dimensional Navier-Stokes equations},
breaklinks=true,
colorlinks=true,
linkcolor=blue,
citecolor=blue,
urlcolor=blue,
filecolor=blue,
}

\numberwithin{equation}{section} 
\setkomafont{title}{\normalfont}

\newenvironment{pdeq}{ \left\{ \begin{aligned}}{\end{aligned}\right.}

\newcommand{\np}[1]{(#1)}
\newcommand{\nb}[1]{[#1]}
\newcommand{\bp}[1]{\big(#1\big)}
\newcommand{\bb}[1]{\big[#1\big]}
\newcommand{\Bp}[1]{\bigg(#1\bigg)}
\newcommand{\Bb}[1]{\bigg[#1\bigg]}
\newcommand{\calj}{{\mathcal J}}

\newcommand{\calp}{{\mathcal P}}

\newcommand{\calt}{{\mathcal T}}

\newcommand{\R}{\mathbb{R}}
\newcommand{\Z}{\mathbb{Z}}
\newcommand{\CNumbers}{\mathbb{C}}
\newcommand{\N}{\mathbb{N}}
\DeclareMathOperator{\e}{e}
\DeclareMathOperator{\B}{B}

\DeclareMathOperator{\id}{Id}
\DeclareMathOperator{\Div}{div}

\DeclareMathOperator{\supp}{supp}
\DeclareMathOperator*{\esssup}{ess\,sup}

\newcommand{\embeds}{\hookrightarrow}

\newcommand{\ra}{\rightarrow}

\newcommand{\set}[1]{\ensuremath{\{#1\}}}

\newcommand{\setc}[2]{\ensuremath{\{#1\ \lvert\ #2\}}}
\newcommand{\setcl}[2]{\ensuremath{\bigl\{#1\ \lvert\ #2\bigr\}}}

\newcommand{\closure}[2]{\overline{#1}^{#2}}
\newcommand{\proj}{\calp}
\newcommand{\projsymbol}{\kappa_0}
\newcommand{\projcompl}{\calp_\bot}
\newcommand{\hproj}{\calp_H}
\newcommand{\quotientmap}{\pi}
\newcommand{\grp}{G}
\newcommand{\dualgrp}{\widehat{G}}

\newcommand{\idmatrix}{I}
\newcommand{\grad}{\nabla}
\newcommand{\gradmap}{\mathrm{grad}}

\newcommand{\dx}{{\mathrm d}x}

\newcommand{\dg}{{\mathrm d}g}
\newcommand{\ds}{{\mathrm d}s}
\newcommand{\dt}{{\mathrm d}t}

\newcommand{\dxi}{{\mathrm d}\xi}

\newcommand{\SR}{\mathscr{S}}

\newcommand{\TDR}{\mathscr{S^\prime}}

\newcommand{\linf}[2]{\langle #1, #2\rangle}

\newcommand{\ft}[1]{\widehat{#1}}
\newcommand{\ift}[1]{#1^\vee}
\newcommand{\FT}{\mathscr{F}}
\newcommand{\iFT}{\mathscr{F}^{-1}}

\newcommand{\norm}[1]{\lVert#1\rVert}
\newcommand{\oseennorm}[2]{\norm{#1}_{#2,\mathrm{Oseen}}}

\newcommand{\normL}[1]{\Bigl\lVert#1\Bigr\rVert}
\newcommand{\snorm}[1]{{\lvert #1 \rvert}}
\newcommand{\snorml}[1]{{\bigl\lvert #1 \big\rvert}}
\newcommand{\snormL}[1]{{\Bigl\lvert #1 \Big\rvert}}

\newcommand{\WSR}[2]{W^{#1,#2}} 
 
\newcommand{\DSR}[2]{D^{#1,#2}} 
 
\newcommand{\WSRloc}[2]{W^{#1,#2}_{loc}} 
\newcommand{\CR}[1]{C^{#1}}  
\newcommand{\LR}[1]{L^{#1}}

\newcommand{\LRloc}[1]{L^{#1}_{loc}} 
\newcommand{\CRi}{\CR \infty}
\newcommand{\CRci}{\CR \infty_0}
\newcommand{\CRc}[1]{\CR{#1}_0}
\newcommand{\LRsigma}[1]{L^{#1}_{\sigma}} 
\newcommand{\gradspace}[1]{\mathscr{G}^{#1}}
 
\newcommand{\DSRNsigma}[2]{D^{#1,#2}_{0,\sigma}} 
 
\newcommand{\WSRsigma}[2]{W^{#1,#2}_{\sigma}} 
\newcommand{\WSRsigmacompl}[2]{W^{#1,#2}_{\sigma,\bot}} 
\newcommand{\LRsigmacompl}[1]{L^{#1}_{\sigma,\bot}}

\newcommand{\CRcisigma}{\CR{\infty}_{0,\sigma}}

\newcommand{\LRper}[1]{L^{#1}_{\mathrm{per}}}
\newcommand{\WSRper}[2]{W^{#1,#2}_{\mathrm{per}}} 
\newcommand{\CRper}{\CR{\infty}_{\mathrm{per}}}

\newcommand{\CRciper}{\CR{\infty}_{0,\mathrm{per}}}
\newcommand{\CRcisigmaper}{\CR{\infty}_{0,\sigma,\mathrm{per}}}
\newcommand{\WSRsigmaper}[2]{W^{#1,#2}_{\sigma,\mathrm{per}}} 
\newcommand{\WSRsigmapercompl}[2]{W^{#1,#2}_{\sigma,\mathrm{per},\bot}} 
\newcommand{\xoseen}[1]{{X}^{#1}_{\sigma,\mathrm{Oseen}}}

\newcommand{\xpres}[1]{{X}^{#1}_{\mathrm{pres}}}

\newcommand{\xspacegeneric}{X}

\newcommand{\nsnonlinb}[2]{#1\cdot\grad #2}
\newcommand{\nsnonlin}[1]{\nsnonlinb{#1}{#1}}
\newcommand{\vvel}{v}

\newcommand{\tvvel}{\tilde{v}}

\newcommand{\wvel}{w}

\newcommand{\twvel}{\tilde{w}}

\newcommand{\uvelnonrel}{\mathfrak{u}}

\newcommand{\uvel}{u}
\newcommand{\uvelft}{\ft{u}}

\newcommand{\upres}{\mathfrak{p}}

\newcommand{\Uvel}{U}
\newcommand{\weakuvel}{\mathcal{U}}
\newcommand{\weakwvel}{\mathcal{W}}
\newcommand{\weakvvel}{\mathcal{V}}
\newcommand{\tweakuvel}{\widetilde{\mathcal{U}}}

\newcommand{\tuvel}{\tilde{u}}
\newcommand{\tupres}{\tilde{\mathfrak{p}}}

\newcommand{\tf}{\tilde{f}}

\newcommand{\ALOseen}{A_{\rm Oseen}}
\newcommand{\ALOseeninverse}{A_{\rm Oseen}^{-1}}
\newcommand{\ALgeneric}{\mathrm{A}}
\newcommand{\ALTP}{{\ALgeneric}_{\mathrm{TP}}}
\newcommand{\ALTPinverse}{{\ALgeneric}^{-1}_{\rm TP}}

\newcommand{\tin}{\text{in }}
\newcommand{\tif}{\text{if }}

\newcommand{\tand}{\text{and }}

\newcommand{\half}{\frac{1}{2}}

\renewcommand{\epsilon}{\varepsilon}
\renewcommand{\phi}{\varphi}
\newcommand{\rey}{\lambda}

\newcommand{\tay}{\calt}
\newcommand{\per}{\tay}
\newcommand{\iper}{\frac{1}{\tay}}
\newcommand{\perf}{\frac{2\pi}{\tay}}
\newcommand{\iperf}{\frac{\tay}{2\pi}}
\newcommand{\eone}{\e_1}

\newcommand{\dualrho}{\hat{\rho}}
\newcommand{\mmultiplier}{m}
\newcommand{\Mmultiplier}{M}

\newcommand{\polynomial}{P}

\newcommand{\tq}{\tilde{q}}

\newcommand{\bijection}{\Pi}
\newcommand{\bijectioninv}{\Pi^{-1}}
\newcommand{\fpmap}{\calj}
\newcommand{\newCCtr}[2][d]{
\newcounter{#2}\setcounter{#2}{0}
\expandafter\xdef\csname kyedtheconst#2\endcsname{#1}
}
\newcommand{\Cc}[2][nolabel]{
\stepcounter{#2}
\expandafter\ensuremath{\csname kyedtheconst#2\endcsname_{\arabic{#2}}}
\ifthenelse{\equal{#1}{nolabel}}
{}
{\expandafter\xdef\csname kyedconst#1\endcsname
{\expandafter\ensuremath{\csname kyedtheconst#2\endcsname_{\arabic{#2}}}}}
}
\newcommand{\CcSetCtr}[2]{
\setcounter{#1}{#2}
}
\newcommand{\Cclast}[1]{
\expandafter\ensuremath{\csname kyedtheconst#1\endcsname_{\arabic{#1}}}
}
\newcommand{\Ccllast}[1]{
\addtocounter{#1}{-1}
\expandafter\ensuremath{\csname kyedtheconst#1\endcsname_{\arabic{#1}}}
\addtocounter{#1}{1}
}
\newcommand{\const}[1]{
\expandafter{\ifcsname kyedconst#1\endcsname
  \csname kyedconst#1\endcsname
\else
  \errmessage{Undefined Kyedconstant #1.}%
\fi}
}

\theoremstyle{plain}
\newtheorem{thm}{Theorem}[section]
\newtheorem{defn}[thm]{Definition}
\newtheorem{lem}[thm]{Lemma}

\newtheorem{cor}[thm]{Corollary}
\theoremstyle{remark}
\newtheorem{rem}[thm]{Remark}

\begin{document}
%%%%%%%%%%%%%%%%%%%%%%%%%%%%%%%%%%%%%%%%%%%%%%%%%%%%%%%%%%%%%%
%%          Title, author, date, abstract, etc.             %%
%%%%%%%%%%%%%%%%%%%%%%%%%%%%%%%%%%%%%%%%%%%%%%%%%%%%%%%%%%%%%%
\title{Existence and regularity of time-periodic solutions to the three-dimensional Navier-Stokes equations}

\author{
Mads Kyed \thanks{Some results in this paper are part of the authors habilitation thesis \cite{habil}.}\\
Institut f\"ur Mathematik \\
Universit\"at Kassel, Germany \\
Email: {\tt mkyed@mathematik.uni-kassel.de}
}

\date{\today}
\maketitle

\begin{abstract}
Existence, uniqueness, and regularity of time-periodic solutions to the Navier-Stokes equations in the three-dimensional whole-space are investigated. 
We consider the Navier-Stokes equations with a non-zero drift term corresponding to the physical model of a fluid flow around a body that moves with a non-zero constant velocity. Existence of a strong time-periodic solution is shown for small time-periodic data. 
It is further shown that this solution is unique in a large class of weak solutions that can
be considered physically reasonable. Finally, we establish regularity properties for any strong solution regardless of its size.
\end{abstract}


%%%%%%%%%%%%%%%%%%%%%%%%%%%%%%%%%%%%%%%%%%%%%%%%%%%%%%%%%%%%%%
%%          Global Constant Counters                        %%
%%%%%%%%%%%%%%%%%%%%%%%%%%%%%%%%%%%%%%%%%%%%%%%%%%%%%%%%%%%%%%
\newCCtr[C]{C}
\newCCtr[M]{M}
\newCCtr[B]{B}
\newCCtr[\epsilon]{eps}
\CcSetCtr{eps}{-1}

%%%%%%%%%%%%%%%%%%%%%%%%%%%%%%%%%%%%%%%%%%%%%%%%%%%%%%%%%%%%%%
%%          Main document                                   %%
%%%%%%%%%%%%%%%%%%%%%%%%%%%%%%%%%%%%%%%%%%%%%%%%%%%%%%%%%%%%%%
\section{Introduction}

We investigate the time-periodic Navier-Stokes equations with a non-zero drift term in the three-dimensional whole-space. More specifically, we consider 
the system
\begin{align}\label{intro_nspastbodywholespace}
\begin{pdeq}
&\partial_t\uvel -\Delta\uvel -\rey\partial_1\uvel + \grad\upres + \nsnonlin{\uvel}  = f && \tin\R^3\times\R,\\
&\Div\uvel =0 && \tin\R^3\times\R
\end{pdeq}
\end{align} 
for an Eulerian velocity field $\uvel:\R^3\times\R\ra\R^3$ and pressure term
$\upres:\R^3\times\R\ra\R$ as well as data $f:\R^3\times\R\ra\R^3$ that are all $\per$-time-periodic, that is, 
\begin{align}\label{intro_timeperiodicsolution}
\begin{aligned}
&\forall (x,t)\in\R^3\times\R:\quad \uvel(x,t) = \uvel(x,t+\per)\quad\tand\quad \upres(x,t) = \upres(x,t+\per)
\end{aligned}
\end{align}
and 
\begin{align}\label{intro_timeperiodicdata}
\begin{aligned}
&\forall (x,t)\in\R^3\times\R:\quad f(x,t) = f(x,t+\per).
\end{aligned}
\end{align}
The time period $\per>0$ is a fixed constant. Physically, the system \eqref{intro_nspastbodywholespace} originates from the model of a flow of an
incompressible, viscous, Newtonian fluid past an object that moves with velocity $\rey\eone\in\R^3$. We shall consider the case $\rey\neq 0$ corresponding to the case
of an object moving with non-zero velocity.

The study of the time-periodic Navier-Stokes equations was initiated by \textsc{Serrin} in \cite{Serrin_PeriodicSolutionsNS1959}. \textsc{Serrin}
postulated that for time-periodic data $f$, and \emph{any} initial value, the solution $\uvelnonrel(x,t)$ to the corresponding initial-value problem 
converges as $t\ra\infty$ to some state which, when considered as an initial value in the initial-value problem, yields a time-periodic solution. 
The rationale behind \textsc{Serrin}'s postulate is that $\uvelnonrel(x,\per n)$, $\per$ being the time period of $f$, 
converges as $n\ra\infty$ to a state on a periodical orbit.
Another approach was introduced independently by \textsc{Yudovich} in \cite{Yudovich60} and \textsc{Prodi} in \cite{Prodi1960}.
These authors proposed to obtain a time-periodic solution by considering the Poincar\'{e} map that takes an initial value into the state obtained 
by evaluation at time $\per$ of the solution to the corresponding initial-value problem, where $\per$ is the period of the prescribed data $f$. 
A time-periodic solution is then identified as a fixed point of this Poincar\'{e} map. 
A further technique based on a representation formula derived from the Stokes semi-group was introduced by \textsc{Kozono} and \textsc{Nakao} in \cite{KozonoNakao96}.
All the methods described above have in common that they utilize the theory for the initial-value problem.
Over the years, a number of investigations based on these methods, or similar ideas involving the initial-value problem in some way, have been carried out:
\cite{Prodi1960},\cite{Prouse1963_TimePer2D},\cite{Prouse63},\cite{KanielShinbrot67},\cite{Takeshita69},\cite{Morimoto1972}, \cite{MiyakawaTeramoto82},\cite{Teramoto1983},\cite{Maremonti_TimePer91},\cite{Maremonti_HalfSpace91},\cite{MaremontiPadula96},\cite{Yamazaki2000},\cite{GaldiSohr2004},
\cite{GaldiSilvestre_Per06},\cite{GaldiSilvestre_Per09},\cite{Taniuchi2009},\cite{BaalenWittwer2011},\cite{Silvestre_TPFiniteKineticEnergy12}. 
None of these papers treat the question of existence and regularity of strong solutions in the case $\rey\neq 0$ of a flow past an obstacle moving with non-zero velocity. Only very recently has this question been investigated by \textsc{Galdi} \cite{GaldiTP2D12,GaldiTP2D13} in the two-dimensional case.

As the main result in this paper, existence of a strong solution to the time-periodic Navier-Stokes equations in the three-dimensional whole-space in the case $\rey\neq 0$ is shown for time-periodic data sufficiently restricted in size, that is, a strong solution to \eqref{intro_nspastbodywholespace}--\eqref{intro_timeperiodicdata}. 
It is further shown that this solution is unique in a large class of weak solutions, that 
it obeys a balance of energy, and that it is as regular as the data allows for. We shall employ a method that does \emph{not} utilize the
corresponding initial-value problem. Instead, we will reformulate \eqref{intro_nspastbodywholespace}--\eqref{intro_timeperiodicdata} as an equivalent system 
on the group $\grp:=\R^3\times\R/\per\Z$. This approach allows us to derive a suitable representation of the solution in terms of a Fourier multiplier based on the
Fourier transform $\FT_\grp$ associated to the group $\grp$. The method allows us to avoid 
the functional analytic setting of the initial-value problem and instead develop one that seems better suited for the time-periodic case.

In \cite{mrtpns} the linearization of \eqref{intro_nspastbodywholespace}--\eqref{intro_timeperiodicdata} was investigated and Banach spaces 
that establish maximal regularity in an $\LR{q}$-setting were identified. The methods deployed in \cite{mrtpns} were based on a
decomposition of the problem into a steady-state problem and a time-periodic problem involving only vector fields with vanishing time-average. We shall
employ in this paper both the decomposition and the maximal regularity results established in \cite{mrtpns}.

\section{Statement of the main results}\label{StatementOfMainResultSection}

We start by defining the function spaces needed to state the main theorems.
We denote points in $\R^3\times\R$ by $(x,t)$, and refer to $x$ as the spatial and $t$ as the time variable. 
We introduce the spaces of real functions
\begin{align*}
&\CRper(\R^3\times\R) := \setcl{U\in\CRi(\R^3\times\R)}{\forall t\in\R:\ U(\cdot,t+\per)=U(\cdot,t)},\\
&\CRciper\bp{\R^3\times[0,\per]} := \setcl{u\in\CRci\bp{\R^3\times[0,\per]}}{\exists U\in\CRper(\R^3\times\R):\ u=U_{|\R^3\times[0,\per]}},\\
&\CRcisigmaper\bp{\R^3\times[0,\per]} := \setcl{u\in\CRciper(\R^3\times[0,\per])^3}{\Div_x u = 0}.
\end{align*}
We introduce Lebesgue and Sobolev spaces as completions of the spaces above in different norms. 
Lebesgue spaces are defined for $q\in[1,\infty)$ by 
\begin{align*}
&\LRper{q}\bp{\R^3\times(0,\per)}:=\closure{\CRciper(\R^3\times[0,\per])}{\norm{\cdot}_{q}},\quad \norm{u}_{q} := \norm{u}_{\LR{q}\bp{\R^3\times(0,\per)}}. 
\end{align*}
Clearly, $\LRper{q}\bp{\R^3\times(0,\per)}$ coincides with the classical Lebesgue space $\LR{q}\bp{\R^3\times(0,\per)}$, and we will therefore omit the subscript $\textrm{per}$ for Lebesgue spaces in the following. 
Sobolev spaces of $\per$-time-periodic functions are defined for $k\in\N_0$ by
\begin{align}\label{MR_DefOfWSRper}
\begin{aligned}
&\WSRper{k}{q}\bp{\R^3\times(0,\per)}:=\closure{\CRciper(\R^3\times[0,\per])}{\norm{\cdot}_{k,q}},\\
&\norm{\uvel}_{k,q}:=\Bp{\sum_{(\alpha,\beta)\in\N_0^3\times\N_0,\ \snorm{(\alpha,\beta)}\leq k} \norm{\partial_t^\beta\partial_x^\alpha\uvel}^q_q }^{1/q}.
\end{aligned}
\end{align}
In contrast to the Lebesgue spaces, the Sobolev spaces $\WSRper{k}{q}\bp{\R^3\times(0,\per)}$ do \emph{not} coincide with the classical Sobolev spaces 
$\WSR{k}{q}\bp{\R^3\times(0,\per)}$.
For $q\in[1,\infty)$ we define the Lebesgue space of solenoidal vector fields
\begin{align*}
&\LRsigma{q}\bp{\R^3\times(0,\per)}:=\closure{\CRcisigmaper(\R^3\times[0,\per])}{\norm{\cdot}_{q}}, 
\end{align*}
and the anisotropic Sobolev space of $\per$-time-periodic, solenoidal, vector fields  
\begin{align}\label{MR_DefOfWSRsigmaper}
\begin{aligned}
&\WSRsigmaper{2,1}{q}\bp{\R^3\times(0,\per)}:= \closure{\CRcisigmaper(\R^3\times[0,\per])}{\norm{\cdot}_{2,1,q}},\\
&\norm{u}_{2,1,q} := 
\Bp{\sum_{(\alpha,\beta)\in\N_0^3\times\N_0,\ \snorm{\alpha}\leq 2,\snorm{\beta}\leq 1} \norm{\partial_x^\alpha u}^q_{q} + 
\norm{\partial_t^\beta u}^q_{q}}^{1/q}. 
\end{aligned}
\end{align}
In order to incorporate the decomposition described in the introduction on the level of function spaces, we define on functions $u:\R^3\times(0,\per)\ra\R$
the operators
\begin{align}\label{intro_defofprojGernericExpression}
\proj u(x,t):=\iper\int_0^\per u(x,s)\,\ds\quad\tand\quad\projcompl u(x,t) := u(x,t)-\proj u(x,t)
\end{align}
whenever these expressions are well-defined. 
Note that $\proj$ and $\projcompl$ decompose a time-periodic vector field $u$ into a time-independent part
$\proj u$ and a time-periodic part $\projcompl u$ with vanishing time-average over the period.
Also note that $\proj$ and $\projcompl$ are complementary projections, that is, $\proj^2=\proj$ and $\projcompl=\id-\proj$. As one may easily verify, 
\begin{align*} 
&\proj,\projcompl:\CRciper(\R^3\times[0,\per])\ra\CRciper(\R^3\times[0,\per]),
\end{align*}
and both projections extend by continuity to
bounded operators on $\LRsigma{q}\bp{\R^3\times(0,\per)}$ and $\WSRsigmaper{2,1}{q}\bp{\R^3\times(0,\per)}$.
We can thus define
\begin{align*}
&\LRsigmacompl{q}\bp{\R^3\times(0,\per)} := \projcompl\LRsigma{q}\bp{\R^3\times(0,\per)},\\
&\WSRsigmapercompl{2,1}{q}\bp{\R^3\times(0,\per)}:= 
\projcompl\WSRsigmaper{2,1}{q}\bp{\R^3\times(0,\per)}.
\end{align*}
For convenience, we introduce for intersections of such spaces the notation
\begin{align*}
&\LRsigmacompl{q,r}\bp{\R^3\times(0,\per)}:=\ \LRsigmacompl{q}\bp{\R^3\times(0,\per)}\ \cap\ \LRsigmacompl{r}\bp{\R^3\times(0,\per)},\\  
&\WSRsigmapercompl{2,1}{q,r}\bp{\R^3\times(0,\per)}:=\ \WSRsigmapercompl{2,1}{q}\bp{\R^3\times(0,\per)}\ \cap\ \WSRsigmapercompl{2,1}{r}\bp{\R^3\times(0,\per)}.
\end{align*}
For $q\in(1,2)$ we let
\begin{align}\label{MR_DefOfxoseenq}
\begin{aligned}
&\xoseen{q}(\R^3):=\setcl{\vvel\in\LRloc{1}(\R^3)^3}{\Div\vvel=0,\ \oseennorm{\vvel}{q}<\infty},\\
&\oseennorm{\vvel}{q} := \snorm{\rey}^{\frac{1}{2}}\norm{\vvel}_{\frac{2q}{2-q}} + \snorm{\rey}^{\frac{1}{4}}\norm{\grad\vvel}_{\frac{4}{4-q}} + \snorm{\rey}\norm{\partial_1\vvel}_q+\norm{\grad^2\vvel}_q,  
\end{aligned}
\end{align}
which is a Banach space intrinsically linked with the three-dimensional Oseen operator. 
For $q\in(1,2)$ and $r\in(1,\infty)$ we put
\begin{align}\label{MR_DefOfxoseenqr}
\begin{aligned}
&\xoseen{q,r}(\R^3):=\setcl{\vvel\in\LRloc{1}(\R^3)^3}{\Div\vvel=0,\ \oseennorm{\vvel}{q,r}<\infty},\\
&\oseennorm{\vvel}{q,r} := \oseennorm{\vvel}{q}+\norm{\grad^2\vvel}_r.
\end{aligned}
\end{align}
For $q\in(1,3)$ and $r\in(1,\infty)$ we further define
\begin{align}\label{MR_DefOfXpres}
\begin{aligned}
&\xpres{q,r}\bp{\R^3\times(0,\per)}:=\setc{\upres\in\LRloc{1}\bp{\R^3\times(0,\per)}}{\norm{\upres}_{\xpres{q,r}}<\infty},\\
&\norm{\upres}_{\xpres{q,r}}:=\Bp{\int_0^\per \norm{\upres(\cdot,t)}_{\frac{3q}{3-q}}^q + \norm{\grad_x\upres(\cdot,t)}_q^q\,\dt }^{1/q}+\norm{\grad_x\upres}_{r}.
\end{aligned}
\end{align}
Finally, we let
\begin{align*}
\DSRNsigma{1}{2}(\R^3):=\overline{\CRcisigma(\R^3)}^{\norm{\grad\cdot}_{2}}=\setc{\uvel\in\LR{6}(\R^3)^3}{\grad\uvel\in\LR{2}(\R^3)^{3\times 3},\ \Div\uvel=0}
\end{align*}
denote the classical homogeneous Sobolev space of solenoidal vector fields (the latter equality above is due to a standard Sobolev embedding theorem) and 
\begin{align*}
\WSRloc{a,b}{q}(\R^3\times\R):=\setc{\uvel\in\LRloc{q}(\R^3\times\R)}{\partial_x^\alpha\uvel,\partial_t^\beta\uvel\in\LRloc{q}(\R^3\times\R)\text{ for }\snorm{\alpha}\leq a,\snorm{\beta}\leq b}
\end{align*}
for $q\in[1,\infty)$ and $a,b\in\N_0$.

Throughout the paper, we shall frequently consider the restriction of $\per$-time-periodic functions defined on $\R^3\times\R$ to the domain $\R^3\times(0,\per)$.
More specifically, without additional notation we implicitly treat $\per$-time-periodic functions $f:\R^3\times\R\ra\R$ as functions $f:\R^3\times(0,\per)\ra\R$.  
If $f$ is independent on $t$, we may treat it as a function $f:\R^3\ra\R$.

We are now in a position to state the main results of the paper. The first theorem establishes existence of a strong solution for sufficiently small data. It it 
further shown that this solution is unique in large class of weak solutions that can be considered physically reasonable. We first define this class.
\begin{defn}\label{UniquenessClassDef}
Let $f\in\LRloc{1}\bp{\R^3\times\R}^3$ satisfy \eqref{intro_timeperiodicdata}.
We say that $\weakuvel\in\LRloc{1}\bp{\R^3\times\R}^3$ satisfying \eqref{intro_timeperiodicsolution} is a \emph{physically reasonable weak time-periodic solution} to \eqref{intro_nspastbodywholespace} if\footnote{We can consider the restriction $\weakuvel\in\LRloc{1}\bp{\R^3\times(0,\per)}$ as a vector-valued mapping $t\ra\weakuvel(\cdot,t)$. Moreover, it is easy to see that $\proj\weakuvel$ and thus also $\projcompl\weakuvel$ are well-defined as 
elements in $\LRloc{1}\bp{\R^3\times(0,\per)}$. Consequently, we may also consider $\projcompl\weakuvel$ as a vector-valued mapping $t\ra\projcompl\weakuvel(\cdot,t)$.}
\begin{enumerate}[1),leftmargin=\parindent, itemindent=0.2cm]
\item\label{UniquenessClassDefProp1} $\weakuvel\in\LR{2}\bp{(0,\per);\DSRNsigma{1}{2}(\R^3)}$,
\item\label{UniquenessClassDefProp2} $\projcompl\weakuvel\in\LR{\infty}\bp{(0,\per);\LR{2}(\R^3)^3}$,
\item\label{UniquenessClassDefProp3} $\weakuvel$ is a generalized $\per$-time-periodic solution to \eqref{intro_nspastbodywholespace} in the sense that for all test functions $\Phi\in\CRcisigmaper\bp{\R^3\times(0,\per)}$ holds
\begin{align}\label{UniquenessClassDefDefofweaksol}
\begin{aligned}
\int_0^\per\int_{\R^3} -\weakuvel\cdot\partial_t\Phi +\grad\weakuvel:\grad\Phi -\rey\partial_1\weakuvel\cdot\Phi + (\nsnonlin{\weakuvel})\cdot\Phi\,\dx\dt  = \int_0^\per\int_{\R^3} f\cdot\Phi\,\dx\dt,
\end{aligned}
\end{align}  
\item\label{UniquenessClassDefProp4} $\weakuvel$ satisfies the energy inequality\footnote{The integral on the right-hand side of \eqref{UniquenessClassDefEnergyIneq} 
is not necessarily well-defined for $f\in\LRloc{1}\bp{\R^3\times\R}^3$ and $\weakuvel$ satisfying \ref{UniquenessClassDefProp1}--\ref{UniquenessClassDefProp2}. Included in the definition of a physically reasonable weak time-periodic solution is therefore an implicit condition that these vector fields possess enough integrability for this integral to be well-defined.}
\begin{align}\label{UniquenessClassDefEnergyIneq}
\begin{aligned}
\int_0^\per\int_{\R^3} \snorm{\grad\weakuvel}^2\,\dx\dt \leq \int_0^\per\int_{\R^3} f\cdot \weakuvel\,\dx\dt. 
\end{aligned}
\end{align}
\end{enumerate}
\end{defn}

\begin{rem}\label{justificationOfPRsol}
The characterization of a solution satisfying \ref{UniquenessClassDefProp1}--\ref{UniquenessClassDefProp4} in Definition \ref{UniquenessClassDef} as a \emph{physically reasonable} solution is justified by the physical properties that can be derived 
from property \ref{UniquenessClassDefProp2} and \ref{UniquenessClassDefProp4}.
More precisely, if we consider the fluid flow corresponding to the Eulerian velocity field $\weakuvel$ as the sum of a steady state $\proj\weakuvel$ and a non-steady part 
$\projcompl\weakuvel$, property \ref{UniquenessClassDefProp2} implies that the kinetic energy of the non-steady part of the flow is bounded.
Property \ref{UniquenessClassDefProp4} states that the energy dissipated due to the viscosity of the fluid is less than the input of 
energy from the external forces. 
\end{rem}

We now state the first main theorem of the paper, which establishes existence of a strong solution unique in the class of \emph{physically reasonable weak solutions}. We shall further show that this solution satisfies an energy equality. The theorem reads:

\begin{thm}\label{ExistenceAndUniquenessThm}
Let $q\in(1,\frac{6}{5}\big]$, $r\in(4,\infty)$ and $\lambda\neq 0$. 
There is a constant $\Cc[ExistenceAndUniquenessThmConst]{eps}>0$ such that for any 
$f\in\LRloc{1}\bp{\R^3\times\R}^3$ satisfying \eqref{intro_timeperiodicdata} and
\begin{align}\label{ExistenceAndUniquenessThmDataCond}
\norm{f}_{\LR{q}\bp{\R^3\times(0,\per)}} + \norm{f}_{\LR{r}\bp{\R^3\times(0,\per)}} \leq \const{ExistenceAndUniquenessThmConst}  
\end{align}
there is a solution $(\uvel,\upres)\in\LRloc{1}\bp{\R^3\times\R}^3\times\LRloc{1}\bp{\R^3\times\R}$ to \eqref{intro_nspastbodywholespace}--\eqref{intro_timeperiodicsolution} with $\uvel=\vvel+\wvel$ and
\begin{align}\label{ExistenceAndUniquenessThmSolSpace}
(\vvel,\wvel,\upres)\in\xoseen{q,r}(\R^3)\times\WSRsigmapercompl{2,1}{q,r}\bp{\R^3\times(0,\per)}\times\xpres{q,r}\bp{\R^3\times(0,\per)}.
\end{align}
Moreover, $\uvel$ belongs to and is unique in the class of \emph{physically reasonable weak solutions} characterized by Definition \ref{UniquenessClassDef},
and it  satisfies the energy equality
\begin{align}\label{EnergyEqEE}
\int_0^\per\int_{\R^3} \snorm{\grad\uvel}^2\,\dx\dt = \int_0^\per\int_{\R^3} f\cdot\uvel\,\dx\dt. 
\end{align}
\end{thm}

The second main theorem of the paper concerns regularity properties of strong solutions. 
More specifically, it is shown that any additional regularity of the data translates into a similar degree of 
additional regularity for the solution. 

\begin{thm}\label{RegularityThm}
Let $q\in\big(1,\frac{4}{3}\big]$, $r\in(8,\infty)$, $\lambda\neq0$ and $m\in\N_0$. If
$f\in\LRloc{1}\bp{\R^3\times\R}^3$ satisfies \eqref{intro_timeperiodicdata} and 
\begin{align*}%\label{xxxUNUSEDLABELxxxRegularityThmDataregulartity}
f\in\WSRper{m}{q}\bp{\R^3\times(0,\per)}^3\cap\WSRper{m}{r}\bp{\R^3\times(0,\per)}^3,
\end{align*}
then a solution $(\uvel,\upres)\in\LRloc{1}\bp{\R^3\times\R}^3\times\LRloc{1}\bp{\R^3\times\R}$ to \eqref{intro_nspastbodywholespace}--\eqref{intro_timeperiodicsolution} in the class \eqref{ExistenceAndUniquenessThmSolSpace} (with $\uvel=\vvel+\wvel$) satisfies
\begin{align*}%\label{xxxUNUSEDLABELxxxRegularityThmFinalRegularity}
\begin{aligned}
&\forall(\alpha,\beta,\kappa)\in\N_0^3\times\N_0^3\times\N_0,\ \snorm{\alpha}\leq m,\ \snorm{\beta}+\snorm{\kappa}\leq m:\\
&\qquad(\partial_x^\alpha\vvel,\partial_x^\beta\partial_t^\kappa\wvel,\partial_x^\beta\partial_t^\kappa\upres)\in 
\WSRloc{2}{r}(\R^3)\times\WSRloc{2,1}{r}\np{\R^3\times\R}\times\WSRloc{1,0}{r}\np{\R^3\times\R}\quad\text{with}\\
&\qquad(\partial_x^\alpha\vvel,\partial_x^\beta\partial_t^\kappa\wvel,\partial_x^\beta\partial_t^\kappa\upres)\in 
\xoseen{q,r}(\R^3)\times\WSRsigmapercompl{2,1}{q,r}\bp{\R^3\times(0,\per)}\times\xpres{q,r}\bp{\R^3\times(0,\per)}.
\end{aligned}
\end{align*}
\end{thm}
 
As a corollary to Theorem \ref{RegularityThm}, we state that a strong solution is smooth if the data is smooth.
\begin{cor}\label{RegularitySmoothnessCor}
Let $q\in\big(1,\frac{4}{3}\big]$, $r\in(8,\infty)$ and $\lambda\neq0$. If
$f\in\CRper\bp{\R^3\times\R}^3$, 
then a solution  
$(\uvel,\upres)\in\LRloc{1}\bp{\R^3\times\R}^3\times\LRloc{1}\bp{\R^3\times\R}$ to \eqref{intro_nspastbodywholespace}--\eqref{intro_timeperiodicsolution} in the class \eqref{ExistenceAndUniquenessThmSolSpace} (with $\uvel=\vvel+\wvel$) satisfies
$\uvel\in\CRper\bp{\R^3\times\R}^3$ and $\upres\in\CRper\bp{\R^3\times\R}$.
\end{cor}

\section{Notation}

Points in $\R^3\times\R$ are denoted by $(x,t)$ with $x\in\R^3$ and $t\in\R$.
We refer to $x$ as the spatial and to $t$ as the time variable. 

For a sufficiently regular function $u:\R^3\times\R\ra\R$, we put $\partial_i u:=\partial_{x_i} u$.
For any multiindex $\alpha\in\N_0^3$, we let $\partial_x^\alpha\uvel:= \sum_{j=1}^3 \partial_j^{\alpha_j}\uvel$
and put $\snorm{\alpha}:=\sum_{j=1}^3 \alpha_j$. Moreover, for $x\in\R^3$ we let 
$x^\alpha:=x_1^{\alpha_1}x_2^{\alpha_2}x_3^{\alpha_3}$.
Differential operators act only in the spatial variable unless otherwise indicated. For example, 
we denote by $\Delta u$ the Laplacian of $u$ with respect to the spatial variables, that is, 
$\Delta u:=\sum_{j=1}^3\partial_j^2 u$. For a vector field $u:\R^3\times\R\ra\R^3$, we 
let $\Div u:=\sum_{j=1}^3\partial_j u_j$ denote the divergence of $u$.
For $u:\R^3\times\R\ra\R^3$ and $v:\R^3\times\R\ra\R^3$ we let $(\nsnonlinb{u}{v}):\R^3\times\R\ra\R^3$
denote the vector field 
$(\nsnonlinb{u}{v})_i:=\sum_{j=1}^3\partial_j v_i u_j$.

For two vectors $a,b\in\R^3$, we let $a \otimes b\in\R^{3\times 3}$ denote the tensor 
with $(a\otimes b)_{ij}:=a_ib_j$. 
We denote by $\idmatrix$ the identity tensor 
$\idmatrix\in\R^{3\times 3}$.

We use the symbol $\embeds$ to denote an embedding $X\embeds Y$ of one vector space $X$ into another vector space $Y$. 
In the case of topological vector spaces, embeddings are always required to be continuous.
For a vector space $X$ and $A,B\subset X$, we write $X=A\oplus B$ iff $A$ and $B$ are subspaces of $X$ with 
$A\cap B=\set{0}$ and $X=A+B$. We also write $a\oplus b$ for elements of $A\oplus B$.

Constants in capital letters in the proofs and theorems are global, while constants in small letters are local to the proof in which they appear.

\section{Reformulation in a group setting}

We let $\grp$ denote the group
\begin{align*}%\label{xxxUNUSEDLABELxxxlt_defofgrp}
\grp:=\R^3\times\R/\per\Z
\end{align*} 
with addition as the group operation. 
Clearly, there is a natural correspondence between $\per$-time-periodic functions defined on $\R^3\times\R$ and functions defined on $\grp$.
We shall take advantage of this correspondence and reformulate
\eqref{intro_nspastbodywholespace}--\eqref{intro_timeperiodicdata} and the main theorems in a setting of vector fields defined on $\grp$. 
For this purpose, we introduce a differentiable structure on $\grp$ and define appropriate Lebesgue and Sobolev spaces.

The group $\grp$, endowed with the canonical topology, is a locally compact abelian group. Consequently, it has a Fourier transform associated to it.
The main advantage of working in a setting of $\grp$-defined functions is the ability to employ this Fourier transform and express solutions to linear systems of 
partial differential equations in terms of Fourier multipliers. 

\subsection{Differentiable structure, distributions and Fourier transform}\label{lt_differentiablestructuresubsection}

The topology and differentiable structure on $\grp$ is inherited from $\R^3\times\R$. More precisely, 
we equip $\grp$ with the quotient topology induced by the canonical quotient mapping
\begin{align*}%\label{xxxUNUSEDLABELxxxlt_quotientmap}
\quotientmap :\R^3\times\R \ra \R^3\times\R/\per\Z,\quad \quotientmap(x,t):=(x,[t]).
\end{align*}
Equipped with the quotient topology, $\grp$ becomes a locally compact abelian group. 
We shall use the restriction 
\begin{align*}
\bijection:\R^3\times[0,\per)\ra\grp,\quad \bijection:=\pi_{|\R^3\times[0,\per)}
\end{align*}
to identify $\grp$ with the domain $\R^3\times[0,\per)$. $\bijection$ is clearly a (continuous) bijection. 

Via $\bijection$, one can identify the Haar measure $\dg$ on $\grp$ as the product of the Lebesgue measure on $\R^3$ and the Lebesgue measure on $[0,\per)$.
The Haar measure is unique up-to a normalization factor, which we choose such that
\begin{align*}
\forall\uvel\in\CRc{}(\grp):\quad \int_\grp \uvel(g)\,\dg = \iper\int_0^\per\int_{\R^3} \uvel\circ\bijection(x,t)\,\dx\dt,
\end{align*}
where $\CRc{}(\grp)$ denotes the space of continuous functions of compact support.  
For the sake of convenience, we will omit the $\bijection$ in integrals with respect to $\dx\dt$ of $\grp$-defined functions, that is, instead of 
$\iper\int_0^\per\int_{\R^3} \uvel\circ\bijection(x,t)\,\dx\dt$ we simply write $\iper\int_0^\per\int_{\R^3} \uvel(x,t)\,\dx\dt$.

Next, we define by
\begin{align}\label{lt_smoothfunctionsongrp}
\CRi(\grp):=\setc{\uvel:\grp\ra\R}{\uvel\circ\quotientmap \in\CRi(\R^3\times\R)}
\end{align}
the space of smooth functions on $\grp$. For $\uvel\in\CRi(\grp)$ we define derivatives 
\begin{align*}%\label{xxxUNUSEDLABELxxxlt_defofgrpderivatives}
\forall(\alpha,\beta)\in\N_0^3\times\N_0:\quad \partial_t^\beta\partial_x^\alpha\uvel := \bb{\partial_t^\beta\partial_x^\alpha (\uvel\circ\quotientmap)}\circ\bijectioninv.
\end{align*}
It is easy to verify for $\uvel\in\CRi(\grp)$ that also $\partial_t^\beta\partial_x^\alpha\uvel\in\CRi(\grp)$. 

With a differentiable structure defined on $\grp$ via \eqref{lt_smoothfunctionsongrp}, we can introduce the space of tempered distributions on $\grp$.
For this purpose, we first recall the Schwartz-Bruhat space of generalized Schwartz functions; see for example \cite{Bruhat61}. More precisely, we define for $\uvel\in\CRi(\grp)$ the semi-norms
\begin{align*}
\rho_{\alpha,\beta,\gamma}(\uvel):=\sup_{(x,t)\in\grp} \snorm{x^\gamma\partial_t^\beta\partial_x^\alpha\uvel(x,t)}\quad
\text{for }(\alpha,\beta,\gamma)\in\N_0^3\times\N_0\times\N_0^3,
\end{align*}
and put 
\begin{align*}
\SR(\grp):=\setc{\uvel\in\CRi(\grp)}{\forall(\alpha,\beta,\gamma)\in\N_0^3\times\N_0\times\N_0^3:\ \rho_{\alpha,\beta,\gamma}(\uvel)<\infty}.
\end{align*}
Clearly, $\SR(\grp)$ is a vector space and $\rho_{\alpha,\beta,\gamma}$ a semi-norm on $\SR(\grp)$. We endow $\SR(\grp)$ with the semi-norm 
topology induced by the family $\setcl{\rho_{\alpha,\beta,\gamma}}{(\alpha,\beta,\gamma)\in\N_0^3\times\N_0\times\N_0^3}$.
The topological dual space $\TDR(\grp)$ of $\SR(\grp)$ is then well-defined. We equip $\TDR(\grp)$ with the weak* topology and refer to it 
as the space of tempered distributions on $\grp$. Observe that both $\SR(\grp)$ and $\TDR(\grp)$ remain closed under multiplication 
by smooth functions that have at most polynomial growth with respect to the spatial variables.

For a tempered distribution $\uvel\in\TDR(\grp)$, distributional derivatives 
$\partial_t^\beta\partial_x^\alpha\uvel\in\TDR(\grp)$ are defined by duality in the usual manner:
\begin{align*}
\forall \psi\in\SR(\grp):\ \linf{\partial_t^\beta\partial_x^\alpha\uvel}{\psi}:=\linf{\uvel}{(-1)^{\snorm{(\alpha,\beta)}}\partial_t^\beta\partial_x^\alpha\psi}. 
\end{align*}
It is easy to verify that $\partial_t^\beta\partial_x^\alpha\uvel$ is well-defined as an element of $\TDR(\grp)$.
For tempered distributions on $\grp$, we keep the convention that differential operators like $\Delta$ and $\Div$ act only in the 
spatial variable $x$ unless otherwise indicated. 

We shall also introduce tempered distributions on $\grp$'s dual group $\dualgrp$. We associate each $(\xi,k)\in\R^3\times\Z$ with 
the character 
$\chi:\grp\ra\CNumbers,\ \chi(x,t):=\e^{ix\cdot\xi+ik\perf t}$
on $\grp$. It is standard to verify that all characters are of this form, and we can thus identify
$\dualgrp = \R^3\times\Z$. By default, $\dualgrp$ is equipped with the compact-open topology, which in this case coincides with the product of the Euclidean topology on $\R^3$ and 
the discrete topology on $\Z$. The Haar measure on $\dualgrp$ is simply the product of the Lebesgue measure on $\R^3$ and the counting measure on $\Z$. 

A differentiable structure on $\dualgrp$ is obtained by introduction of the space 
\begin{align*}
\CRi(\dualgrp):=\setc{\wvel\in\CR{}(\dualgrp)}{\forall k\in\Z:\ \wvel(\cdot,k)\in\CRi(\R^3)}.
\end{align*}
To define the generalized Schwartz-Bruhat space on the dual group $\dualgrp$, we further introduce for $\wvel\in\CRi(\dualgrp)$ the semi-norms
\begin{align*}
\dualrho_{\alpha,\beta,\gamma}(\wvel):= 
\sup_{(\xi,k)\in\dualgrp} \snorm{k^\beta \xi^\alpha \partial_\xi^\gamma \wvel(\xi,k)}
\quad\text{for }(\alpha,\beta,\gamma)\in\N_0^3\times\N_0\times\N_0^3.
\end{align*}
We then put 
\begin{align*}
\begin{aligned} 
\SR(\dualgrp)&:=\setc{\wvel\in\CRi(\dualgrp)}{\forall (\alpha,\beta,\gamma)\in\N_0^3\times\N_0\times\N_0^3:\ \dualrho_{\alpha,\beta,\gamma}(\wvel)<\infty}.
\end{aligned}
\end{align*}
We endow the vector space $\SR(\dualgrp)$ with the semi-norm topology induced by the family
of semi-norms $\setc{\dualrho_{\alpha,\beta,\gamma}}{(\alpha,\beta,\gamma)\in\N_0^3\times\N_0\times\N_0^3}$. The topological dual space  
of $\SR(\dualgrp)$ is denoted by $\TDR(\dualgrp)$. We equip $\TDR(\dualgrp)$ with the weak* topology and refer to it 
as the space of tempered distributions on $\dualgrp$.

So far, all function spaces have been defined as real vector spaces of real functions. 
Clearly, we can define them analogously as complex vector spaces of complex functions. 
When a function space is used in context with the Fourier transform, which we shall 
introduce below, we consider it as a complex vector space. 

The Fourier transform $\FT_\grp$ on $\grp$ is given by
\begin{align*}
\FT_\grp:\LR{1}(\grp)\ra\CR{}(\dualgrp),\quad \FT_\grp(\uvel)(\xi,k):=\ft{\uvel}(\xi,k):=
\iper\int_0^\per\int_{\R^3} \uvel(x,t)\,\e^{-ix\cdot\xi-ik\perf t}\,\dx\dt.
\end{align*}
If no confusion can arise, we simply write $\FT$ instead of $\FT_\grp$.
The inverse Fourier transform is formally defined by 
\begin{align*}
\iFT:\LR{1}(\dualgrp)\ra\CR{}(\grp),\quad \iFT(\wvel)(x,t):=\ift{\wvel}(x,t):=
\sum_{k\in\Z}\,\int_{\R^3} \wvel(\xi,k)\,\e^{ix\cdot\xi+ik\perf t}\,\dxi.
\end{align*}
It is standard to verify that $\FT:\SR(\grp)\ra\SR(\dualgrp)$ is a homeomorphism with $\iFT$ as the actual inverse, provided the Lebesgue measure $\dxi$ is normalized appropriately. 
By duality, $\FT$ extends to a mapping $\TDR(\grp)\ra\TDR(\dualgrp)$. More precisely, we define 
\begin{align*}
\FT:\TDR(\grp)\ra\TDR(\dualgrp),\quad \forall\psi\in\SR(\dualgrp):\ \linf{\FT(\uvel)}{\psi}:=\linf{\uvel}{\FT({\psi})}.
\end{align*}
Similarly, we define
\begin{align*}
\iFT:\TDR(\dualgrp)\ra\TDR(\grp),\quad \forall\psi\in\SR(\grp):\ \linf{\iFT(\uvel)}{\psi}:=\linf{\uvel}{\iFT({\psi})}.
\end{align*}
Clearly $\FT:\TDR(\grp)\ra\TDR(\dualgrp)$ is a homeomorphism with $\iFT$ as the actual inverse.

The Fourier transform in the setting above provides us with a calculus between the differential operators on $\grp$ and the 
polynomials on $\dualgrp$. As one easily verifies, for $\uvel\in\TDR(\grp)$ and $\alpha\in\N_0^3$, $l\in\N_0$ we have
\begin{align*}
\FT\bp{\partial_t^l\partial_x^\alpha\uvel}=i^{l+\snorm{\alpha}}\,\Big(\perf\Big)^l\,k^l\,\xi^\alpha\,\FT(\uvel)
\end{align*} 
as identity in $\TDR(\dualgrp)$.

\subsection{Function spaces}\label{lt_functionspacesSection}

Having introduced smooth functions on $\grp$ in form of the space $\CRi(\grp)$, we define function spaces of $\grp$-defined functions and vector fields corresponding
to the Lebesgue and Sobolev spaces of $\per$-time-periodic functions and vector fields introduced in Section \ref{StatementOfMainResultSection}. 

We start by putting 
\begin{align*}
&\CRci(\grp):=\setc{\uvel\in\CRi(\grp)}{\supp\uvel\text{ is compact}}.
\end{align*}
We let $\LR{q}(\grp)$ denote the usual Lebesgue space with respect to the Haar measure $\dg$, and let $\norm{\cdot}_q$ denote the norm.
It is standard to verify that $\LR{q}(\grp)\subset\TDR(\grp)$.
Classical Sobolev spaces are then defined as  
\begin{align*}
\WSR{k}{q}(\grp):=\setc{\uvel\in\LR{q}(\grp)}{\norm{\uvel}_{k,q}<\infty},
\end{align*}
where $\norm{\cdot}_{k,q}$ is defined exactly as in \eqref{MR_DefOfWSRper}, and the condition $\norm{\uvel}_{k,q}<\infty$ expresses that the distributional derivatives
of $\uvel$ appearing in the norm $\norm{\cdot}_{k,q}$ all belong to $\LR{q}(\grp)$. We note that 
$\WSR{k}{q}(\grp)=\closure{\CRci(\grp)}{\norm{\cdot}_{k,q}}$ and $\LR{q}(\grp)=\closure{\CRci(\grp)}{\norm{\cdot}_{q}}$,
which can be shown by standard arguments.

Next, we let 
\begin{align*}
&\CRcisigma(\grp):=\setc{\uvel\in\CRci(\grp)^3}{\Div\uvel=0},
\end{align*}
and define the Banach spaces 
\begin{align*}
\LRsigma{q}(\grp):=\closure{\CRcisigma(\grp)}{\norm{\cdot}_{q}},\quad
\WSRsigma{2,1}{q}(\grp):= \closure{\CRcisigma(\grp)}{\norm{\cdot}_{2,1,q}}
\end{align*}
of solenoidal vector fields, where the norm $\norm{\cdot}_{2,1,q}$ is defined as in \eqref{MR_DefOfWSRsigmaper}.
It can be shown that
\begin{align}
&\LRsigma{q}(\grp)=\setc{\uvel\in\LR{q}(\grp)^3}{\Div\uvel=0}.\label{lt_densitylemmaLRsigmaCharacterization}
\end{align}
This identity is well-known if the underlying domain is $\R^3$; a proof can be found in \cite[Chapter III.4]{galdi:book1}. 
Simple modifications to this proof (see \cite[Lemma 3.2.1]{habil}) suffice to establish the identity in the case where $\R^3$ is replaced with $\grp$. 
For convenience, we put 
\begin{align*}
&\LRsigma{q,r}(\grp):=\LRsigma{q}(\grp)\cap\LRsigma{r}(\grp),&& \norm{\cdot}_{\LRsigma{q,r}(\grp)}:=\norm{\cdot}_q+\norm{\cdot}_r,\\
&\WSRsigma{2,1}{q,r}(\grp):=\WSRsigma{2,1}{q}(\grp)\cap\WSRsigma{2,1}{r}(\grp),&& \norm{\cdot}_{2,1,q,r}:=\norm{\cdot}_{2,1,q}+\norm{\cdot}_{2,1,r},
\end{align*}
which are obviously Banach spaces in the associated norms.

Recalling \eqref{intro_defofprojGernericExpression}, we define analogously the projection $\proj$  on $\grp$-defined functions:
\begin{align*}%\label{xxxUNUSEDLABELxxxDefProjInGrpSetting}
\proj:\CRci(\grp)\ra\CRci(\grp),\quad \proj u(x,t):=\iper\int_0^\per u(x,s)\,\ds 
\end{align*}
and put $\projcompl:=\id-\proj$.
We make note of the following properties:
\begin{lem}\label{ProjInGrpSettingLem}
Let $q\in(1,\infty)$. The projection $\proj:\CRcisigma(\grp)\ra\CRcisigma(\grp)$ 
extends, by continuity, uniquely to a bounded projection $\proj:\LRsigma{q}(\grp)\ra\LRsigma{q}(\grp)$ and to a bounded projection
$\proj:\WSRsigma{2,1}{q}(\grp)\ra\WSRsigma{2,1}{q}(\grp)$. The same is true for $\projcompl$.
\end{lem}
\begin{proof}
Boundedness of $\proj$ in the norms of $\LRsigma{q}(\grp)$ and $\WSRsigma{2,1}{q}(\grp)$ can easily be verified by employing H\"older's and Minkowski's integral
inequality; see also \cite[Lemma 4.5]{mrtpns}.
\end{proof}

\begin{lem}\label{lt_projextensionl1loc}
$\proj$ extends uniquely to a projection $\proj:\LRloc{1}(\grp)\ra\LRloc{1}(\grp)$. The same is true for 
$\projcompl$. 
\end{lem}
\begin{proof}
For any $R>0$, $\proj$ extends uniquely, by continuity, to a bounded projection on $\LR{1}(\B_R\times\R/\per\Z)$. Thus, 
for $\uvel\in\LRloc{1}(\grp)$ the element $\proj\uvel$ is naturally defined in $\LR{1}(\B_R\times\R/\per\Z)$ for any $R>0$.
Consequently, $\proj\uvel$, and thus also $\projcompl\uvel$, are well-defined as elements in $\LRloc{1}(\grp)$. 
\end{proof}

\begin{lem}\label{lt_projprojcomplinnerprodl1locfunctions}
Let $f,g\in\LRloc{1}(\grp)$. Then 
\begin{align}\label{lt_projprojcomplinnerprodl1locfunctionsinnerprod}
\begin{aligned}
\iper\int_0^\per \proj f(x,t) \cdot \projcompl g(x,t)\,\dt=0 \quad\text{for a.e. }x\in\R^3.
\end{aligned}
\end{align}
\end{lem}
\begin{proof}
This is a simple consequence of the fact that $\proj f$ is independent on $t$.
\end{proof}

\begin{lem}\label{lt_projsymbollem} 
The projections $\proj$ and $\projcompl$ extend uniquely, by continuity, to continuous operators $\proj:\TDR(\grp)\ra\TDR(\grp)$ and $\projcompl:\TDR(\grp)\ra\TDR(\grp)$ with
\begin{align}
&\proj f = \iFT_\grp\bb{\projsymbol\cdot \ft{f}},\label{lt_projsymbollemproj}\\
&\projcompl f = \iFT_\grp\bb{(1-\projsymbol)\cdot \ft{f}}\label{lt_projsymbollemprojcompl},
\end{align}
where
\begin{align*}
\projsymbol:\dualgrp\ra\CNumbers,\quad
\projsymbol(\xi,k):=
\begin{pdeq}
&1 && \tif k=0,\\
&0 && \tif k\neq0.
\end{pdeq}
\end{align*}
\end{lem}
\begin{proof}
We simply observe for $f\in\SR(\grp)$ that
\begin{align*}
\FT_\grp\bb{\proj f}(\xi,k) &= \iper\int_0^\per\int_{\R^n}\iper\int_0^\per f(x,s)\,\ds\,e^{-ix\cdot\xi-i\perf k t}\,\dx\dt\\
&= \projsymbol(\xi,k) \int_{\R^n}\iper\int_0^\per f(x,s)\,\ds\,e^{-ix\cdot\xi}\,\dx
= \projsymbol(\xi,k)\, \ft{f}(\xi,0) = \projsymbol(\xi,k)\, \ft{f}(\xi,k).
\end{align*}
The formula extends to $f\in\TDR(\grp)$ by duality.
\end{proof}

Having introduced the projections $\proj$ and $\projcompl$, we can now define
\begin{align*}
&\LRsigmacompl{q}(\grp) := \projcompl\LRsigma{q}(\grp),\\
&\LRsigmacompl{q,r}(\grp) := \projcompl\LRsigma{q,r}(\grp)\ = \LRsigmacompl{q}(\grp)  \cap\ \LRsigmacompl{r}(\grp),\\
&\WSRsigmacompl{2,1}{q}(\grp):= \projcompl\WSRsigma{2,1}{q}(\grp),\\
&\WSRsigmacompl{2,1}{q,r}(\grp):= \projcompl\WSRsigmacompl{2,1}{q,r}(\grp) = \WSRsigmacompl{2,1}{q}(\grp)\ \cap\ \WSRsigmacompl{2,1}{r}(\grp).
\end{align*}
Since $\proj u$ is $t$-independent, it is easy to verify that $\proj\LRsigma{q}(\grp)=\LRsigma{q}(\R^3)$. It follows that $\proj$ induces the decomposition
\begin{align}\label{lrsigmaDecomp}
\LRsigma{q}(\grp)=\LRsigma{q}(\R^3)\oplus\LRsigmacompl{q}(\grp).
\end{align}

Next, we introduce the Helmholtz projection on the Lebesgue space $\LR{q}(\grp)^3$ by a classical Fourier-multiplier expression: 
\begin{lem}\label{lt_HelmholtzProjDefLem}
The Helmholtz projection 
\begin{align}\label{lt_HelmholtzProjDefDef}
\hproj: \LR{2}(\grp)^3\ra\LR{2}(\grp)^3,\quad \hproj f := \iFT_\grp\Bb{\Bp{\idmatrix - \frac{\xi\otimes\xi}{\snorm{\xi}^2}} \ft{f}}
\end{align}
extends for any $q\in[1,\infty)$ uniquely to a continuous projection $\hproj:\LR{q}(\grp)^3\ra\LR{q}(\grp)^3$. Moreover,
$\hproj\LR{q}(\grp)^3=\LRsigma{q}(\grp)$.
\end{lem}
\begin{proof}
The Fourier multiplier on the right-hand side in \eqref{lt_HelmholtzProjDefDef} is identical to the multiplier of the classical Helmholtz projection 
in the Euclidean $\R^3$-setting. Boundedness of $\hproj$ on $\LR{q}(\grp)^3$ can thus be derived from boundedness of the classical Helmholtz projection 
on $\LR{q}(\R^3)^3$. One readily verifies that $\hproj$ is a projection,
and that $\Div\hproj f=0$. By \eqref{lt_densitylemmaLRsigmaCharacterization}, $\hproj\LR{q}(\grp)^3\subset\LRsigma{q}(\grp)$ follows. On the other hand, 
since $\Div f=0$ implies $\xi_j\ft{f}_j=0$, we have $\hproj f=f$ for all $f\in\LRsigma{q}(\grp)$. We conclude $\hproj\LR{q}(\grp)^3=\LRsigma{q}(\grp)$. 
\end{proof}

Since $\hproj:\LR{q}(\grp)^3\ra\LR{q}(\grp)^3$ is a continuous projection, it decomposes $\LR{q}(\grp)$ into a direct sum 
\begin{align*}%\label{xxxUNUSEDLABELxxxlt_HelmholtzDecomp}
\LR{q}(\grp)=\LRsigma{q}(\grp)\oplus \gradspace{q}(\grp)
\end{align*}
of closed subspaces with
\begin{align*}%\label{xxxUNUSEDLABELxxxlt_gradspacedef}
\gradspace{q}(\grp) := \bp{\id-\hproj}\LR{q}(\grp)^3.
\end{align*}
We further define
\begin{align*}%\label{xxxUNUSEDLABELxxxss_kspacedefsLRsigma}
&\gradspace{q,r}(\grp):=\gradspace{q}(\grp)\cap\gradspace{q}(\grp),\quad \norm{\cdot}_{\gradspace{q,r}(\grp)}:=\norm{\cdot}_q+\norm{\cdot}_r,
\end{align*}
which is clearly a Banach space with respect to the associated norm.

We introduce the convention that a $\grp$-defined function $\uvel:\grp\ra\R$ can be considered an element of a function space of $\R^3$-defined functions, say
$\xspacegeneric(\R^3)$, if and only if $\uvel$ is independent on $t$, and the restriction $u_{|\R^3\times\set{0}}$ belongs to $\xspacegeneric(\R^3)$.
In this context, we shall need, in addition to the spaces $\xoseen{q,r}(\R^3)$ defined in \eqref{MR_DefOfxoseenqr}, also the homogeneous Sobolev spaces $\DSR{m}{q}(\R^3)$ and their associated semi-norms:
\begin{align*}
\begin{aligned}
&\DSR{m}{q}(\R^3):=\setc{u\in\LRloc{1}(\R^3)}{\forall\alpha\in\N_0^3\text{ with } \snorm{\alpha}=m:\ \partial^\alpha u \in\LR{q}(\R^3)},\\
&\snorm{u}_{m,q}:=\bigg( \sum_{\snorm{\alpha}=m}\, \int_{\R^3} \snorm{\partial^\alpha u(x)}^q\,\dx \bigg)^\frac{1}{q}.
\end{aligned}
\end{align*}
Moreover, we will deploy the space of solenoidal vector fields 
\begin{align*}
\LRsigma{q}(\R^3):=\closure{\CRcisigma(\R^3)}{\norm{\cdot}_q}
\end{align*}
and
\begin{align*}
\LRsigma{q,r}(\R^3):=\LRsigma{q}(\R^3)\cap\LRsigma{r}(\R^3),\quad \norm{\cdot}_{\LRsigma{q,r}(\R^3)}:=\norm{\cdot}_q+\norm{\cdot}_r\,,
\end{align*}
which is obviously a Banach space in the given norm.

Finally, we define for $\grp$-defined functions the norm $\norm{\cdot}_{\xpres{q,r}}$ exactly as in \eqref{MR_DefOfXpres} and let
\begin{align*}
\xpres{q,r}(\grp):=\setc{\upres\in\LRloc{1}(\grp)}{\norm{\upres}_{\xpres{q,r}}<\infty}.
\end{align*}

\subsection{Reformulation}

Since the differentiable structure on $\grp$ is inherited from $\R^3\times\R$, we can formulate \eqref{intro_nspastbodywholespace}--\eqref{intro_timeperiodicdata}
as a system of partial differential equations on $\grp$:
\begin{align}\label{ss_nsongrp}
\begin{pdeq}
&\partial_t\uvel -\Delta\uvel -\rey\partial_1\uvel + \grad\upres + \nsnonlin{\uvel}= f && \tin\grp,\\
&\Div\uvel =0 && \tin\grp,
\end{pdeq}
\end{align} 
with unknowns $\uvel:\grp\ra\R^3$ and $\upres:\grp\ra\R$, and data $f:\grp\ra\R^3$.
Observe that in this formulation the periodicity conditions are 
not needed anymore. 
Indeed, all functions defined on $\grp$ are by construction $\per$-time-periodic.

Based on the new formulation above, we obtain the following new formulations of
Theorem \ref{ExistenceAndUniquenessThm} and Theorem \ref{RegularityThm} in a setting of $\grp$-defined vector fields. 
For convenience, in the new formulation we split Theorem \ref{ExistenceAndUniquenessThm} into three parts: the statement of existence, the balance of energy, and 
the statement of uniqueness. 

\begin{thm}\label{ss_StrongSolThm}
Let $q\in(1,\frac{4}{3}\big]$, $r\in(4,\infty)$ and $\lambda\neq 0$. 
There is a constant $\Cc[ss_StrongSolThmEps]{eps}>0$ such that for any 
$f\in\LR{q}(\grp)^3\cap\LR{r}(\grp)^3$ with
\begin{align}\label{ss_StrongSolThmDataCond}
\norm{f}_{\LR{q}(\grp)} + \norm{f}_{\LR{r}(\grp)} \leq \const{ss_StrongSolThmEps}  
\end{align}
there is a solution $(\uvel,\upres)$ to \eqref{ss_nsongrp} with $\uvel=\vvel+\wvel$ and
\begin{align}\label{ss_StrongSolThmSolSpace}
(\vvel,\wvel,\upres)\in\xoseen{q,r}(\R^3)\times\WSRsigmacompl{2,1}{q,r}(\grp)\times\xpres{q,r}(\grp).
\end{align}
\end{thm}

\begin{defn}\label{ss_UniquenessClassDef}
Let $f\in\LRloc{q}(\grp)^3$.
We say that $\weakuvel\in\LRloc{1}(\grp)^3$ is a \emph{physically reasonable weak solution} to \eqref{ss_nsongrp}
if, considered as a mapping $t\ra\weakuvel(\cdot,t)$, it satisfies  
$\weakuvel\in\LR{2}\bp{(0,\per);\DSRNsigma{1}{2}(\R^3)}$,
$\projcompl\weakuvel\in\LR{\infty}\bp{(0,\per);\LR{2}(\R^3)^3}$,
$\weakuvel$ satisfies \eqref{UniquenessClassDefDefofweaksol} for all $\Phi\in\CRcisigma(\grp)$, and
$\weakuvel$ satisfies the energy inequality \eqref{UniquenessClassDefEnergyIneq}.
\end{defn}

\begin{thm}\label{ss_EnergyEqThm}
Let $q\in\big(1,\frac{4}{3}\big]$, $r\in(4,\infty)$, $\rey\neq 0$ and $f\in\LR{q}(\grp)^3\cap\LR{r}(\grp)^3$.
A solution $(\uvel,\upres)$ to \eqref{ss_nsongrp} in the class \eqref{ss_StrongSolThmSolSpace} (with $\uvel=\vvel+\wvel$)
satisfies the energy equation \eqref{EnergyEqEE}.
\end{thm}

\begin{thm}\label{ss_UniquenessThm}
Let $q\in(1,\frac{6}{5}\big]$, $r\in(4,\infty)$, $\lambda\neq 0$ and $f\in\LR{q}(\grp)^3\cap\LR{r}(\grp)^3$. There is a constant 
$\Cc[ss_UniquenessThmConst]{eps}>0$ such that if 
$\norm{f}_q + \norm{f}_r \leq \const{ss_UniquenessThmConst}$,  
then a solution $(\uvel,\upres)$ to \eqref{ss_nsongrp} in the class \eqref{ss_StrongSolThmSolSpace} (with $\uvel=\vvel+\wvel$) 
is unique in the class of \emph{physically reasonable weak solutions} characterized by Definition \ref{ss_UniquenessClassDef}.
\end{thm}

\begin{thm}\label{ss_RegularityThm}
Let $q\in\big(1,\frac{4}{3}\big]$, $r\in(8,\infty)$, $\lambda\neq0$ and $m\in\N$. If
$f\in\WSR{m}{q}(\grp)^3\cap\WSR{m}{r}(\grp)^3$,
then a solution $(\uvel,\upres)$ to \eqref{ss_nsongrp} in the class \eqref{ss_StrongSolThmSolSpace} (with $\uvel=\vvel+\wvel$) satisfies
\begin{align*}%\label{xxxUNUSEDLABELxxxss_RegularityThmFinalRegularity}
\begin{aligned}
&\forall(\alpha,\beta,\kappa)\in\N_0^3\times\N_0^3\times\N_0,\ \snorm{\alpha}\leq m,\ \snorm{\beta}+\snorm{\kappa}\leq m:\\
&\qquad(\partial_x^\alpha\vvel,\partial_x^\beta\partial_t^\kappa\wvel,\partial_x^\beta\partial_t^\kappa\upres)\in 
\xoseen{q,r}(\R^3)\times\WSRsigmacompl{2,1}{q,r}(\grp)\times\xpres{q,r}(\grp).
\end{aligned}
\end{align*}
\end{thm}

The main challenge will now be to prove the theorems above, which will be done in the next section.
The advantage obtained at this point, by the reformulation of these theorems in the setting on the group $\grp$,
is the ability by means of the Fourier transform $\FT_\grp$ to employ multiplier theory.

\section{Proof of main theorems}

We will now prove Theorem \ref{ExistenceAndUniquenessThm} and Theorem \ref{RegularityThm}.
The proofs reduce to simple verifications once we have established Theorem \ref{ss_StrongSolThm}--Theorem \ref{ss_RegularityThm}. 
First, however, we recall the maximal regularity results for the linearization of \eqref{ss_nsongrp} from \cite{mrtpns}. 
Based on the linear theory, the existence of a strong solution as stated in Theorem \ref{ss_StrongSolThm} will then be shown with the contraction mapping
principle. Also the regularity properties in Theorem \ref{ss_RegularityThm} will be established from estimates obtained for the linear problem.

We shall repeatedly make use of the following embedding property of homogeneous Sobolev spaces:
\begin{lem}\label{HomSobEmbHighq}
Let $q\in(1,\infty)$ and $r\in(3,\infty)$. Then 
\begin{align}\label{HomSobEmbHighqEmb}
\forall u\in\DSR{1}{r}(\R^3)\cap\LR{q}(\R^3):\quad \norm{u}_\infty \leq \Cc[HomSobEmbHighqConst]{C} \bp{\snorm{u}_{1,r}+\norm{u}_q}
\end{align}
with $\Cclast{C}=\Cclast{C}(q,r)$.	
\end{lem}
\begin{proof}
See \cite[Remark II.7.2]{galdi:book1}.
\end{proof}

\begin{lem}\label{ss_StrongSolThmEmbeddingLem}
Let $q\in\big(1,\frac{4}{3}\big]$ and $r\in(4,\infty)$. Then every $\vvel\in\xoseen{q,r}(\R^3)$ satisfies
\begin{align}
&\norm{\grad\vvel}_\infty\leq \Cc[ss_StrongSolThmEmbeddingLemConstGradvinfty]{C}\norm{\vvel}_{\xoseen{q,r}(\R^3)},\label{ss_StrongSolThmEmbeddingLemGradvinfty}\\
&\norm{\grad\vvel}_r\leq \Cc[ss_StrongSolThmEmbeddingLemConstGradvr]{C}\norm{\vvel}_{\xoseen{q,r}(\R^3)},\label{ss_StrongSolThmEmbeddingLemGradvr}\\
&\norm{\vvel}_\infty\leq \Cc[ss_StrongSolThmEmbeddingLemConstvinfty]{C}\norm{\vvel}_{\xoseen{q,r}(\R^3)},
\label{ss_StrongSolThmEmbeddingLemvinfty}\\
&\norm{\grad\vvel}_2\leq \Cc[ss_StrongSolThmEmbeddingLemConstGradv2]{C}\norm{\vvel}_{\xoseen{q,r}(\R^3)}.
\label{ss_StrongSolThmEmbeddingLemGradv2}
\end{align}
Moreover, every $\wvel\in\WSRsigmacompl{2,1}{q,r}(\grp)$ satisfies
\begin{align}
\norm{\wvel}_\infty \leq \Cc[ss_StrongSolThmEmbeddingLemConstwinfty]{C}\norm{\wvel}_{2,1q,r}.\label{ss_StrongSolThmEmbeddingLemwinfty}
\end{align}
\end{lem}
\begin{proof}\newCCtr[c]{ss_StrongSolThmEmbeddingLem}
Recall $\eqref{HomSobEmbHighqEmb}$ and observe that 
\begin{align*}
\norm{\grad\vvel}_\infty\leq \const{HomSobEmbHighqConst}\bp{\snorm{\vvel}_{2,r}+\norm{\grad\vvel}_{\frac{4q}{4-q}}}\leq 
\const{HomSobEmbHighqConst}\norm{\vvel}_{\xoseen{q,r}(\R^3)},
\end{align*}
which implies \eqref{ss_StrongSolThmEmbeddingLemGradvinfty}.
It follows that $\grad\vvel\in\LR{\frac{4q}{4-q}}(\R^3)\cap\LR{\infty}(\R^3)$ and consequently, since $\frac{4q}{4-q}<r<\infty$, by interpolation that
\begin{align*}
\norm{\grad\vvel}_r \leq \Cc{ss_StrongSolThmEmbeddingLem}\bp{\norm{\grad\vvel}_\infty + \norm{\grad\vvel}_\frac{4q}{4-q} } 
\leq \Cclast{ss_StrongSolThmEmbeddingLem} \norm{\vvel}_{\xoseen{q,r}(\R^3)}.
\end{align*}
This shows \eqref{ss_StrongSolThmEmbeddingLemGradvr}.
With \eqref{ss_StrongSolThmEmbeddingLemGradvr} at our disposal, we again employ \eqref{HomSobEmbHighqEmb} and find that
\begin{align*}
\norm{\vvel}_\infty \leq \const{HomSobEmbHighqConst}\bp{\norm{\grad\vvel}_r + \norm{\vvel}_\frac{2q}{2-q} } 
\leq \Cc{ss_StrongSolThmEmbeddingLem} \norm{\vvel}_{\xoseen{q,r}(\R^3)}.
\end{align*}
Thus \eqref{ss_StrongSolThmEmbeddingLemvinfty} follows. To show \eqref{ss_StrongSolThmEmbeddingLemGradv2}, 
observe, since $q\leq\frac{4}{3}$ and thus $\frac{4q}{4-q}\leq 2$, that
\begin{align*}
\norm{\grad\vvel}_2 \leq \Cc{ss_StrongSolThmEmbeddingLem}\bp{\norm{\grad\vvel}_{\frac{4q}{4-q}} + \norm{\grad\vvel}_{\infty} } 
\leq \Cc{ss_StrongSolThmEmbeddingLem} \norm{\vvel}_{\xoseen{q,r}(\R^3)}.
\end{align*} 
Finally, the Sobolev embedding $\WSR{1}{r}(\grp)\embeds \LR{\infty}(\grp)$ for $r>4$, which follows from the classical Sobolev embedding
$\WSR{1}{r}\bp{\R^3\times(0,\per)}\embeds \LR{\infty}\bp{\R^3\times(0,\per)}$ since $\bijection$, by lifting, induces an embedding
$\WSR{1}{r}(\grp){\embeds}\WSR{1}{r}\bp{\R^3\times(0,\per)}$,
implies \eqref{ss_StrongSolThmEmbeddingLemwinfty}.
\end{proof}

\begin{lem}\label{lt_TPOseenUniqueness}
If $\uvel\in\TDR(\grp)$ with $\proj\uvel=0$ satisfies
\begin{align}\label{lt_TPOseenUniquenessEquation}
\partial_t\uvel -\Delta\uvel -\rey\partial_1\uvel = 0\quad \tin\grp,
\end{align} 
then $\uvel=0$.
\end{lem}
\begin{proof}
Applying $\FT_\grp$ on both sides in \eqref{lt_TPOseenUniquenessEquation}, we deduce that 
$\bp{i\perf k + \snorm{\xi}^2 -\rey i\xi_1} \uvelft = 0$.
Since the polynomial $\snorm{\xi}^2 + i\bp{\perf k -\rey \xi_1}$ vanishes only at $(\xi,k)=(0,0)$,
we conclude that $\supp\ft{\uvel}\subset\set{(0,0)}$. However, since 
$\proj\uvel=0$ we have $\projsymbol\ft{\uvel}=0$, whence $(\xi,0)\notin\supp\ft{\uvel}$ for all $\xi\in\R^3$. Consequently,
$\supp\ft{\uvel}=\emptyset$. It follows that $\ft{\uvel}=0$ and thus $\uvel=0$. 
\end{proof}

\begin{lem}\label{lt_OseenUniquenessLem}
Let $\vvel\in\LR{q}(\R^3)$ for some $q\in[1,\infty)$. If 
\begin{align}\label{lt_OseenUniquenessLemEquation}
-\Delta\vvel -\rey\partial_1\vvel = 0\quad \tin\R^3,
\end{align} 
then $\vvel=0$.
\end{lem}
\begin{proof}
Applying the Fourier transform $\FT_{\R^3}$ in \eqref{lt_OseenUniquenessLemEquation}, we see that
$\bp{\snorm{\xi}^2 -\rey i\xi_1} \ft{\vvel} = 0$.
It follows that $\supp\ft{\vvel}\subset\set{0}$, whence $\vvel$ is a polynomial. Since $\vvel\in\LR{q}(\R^3)$, we must have
$\vvel=0$.
\end{proof}

\begin{lem}\label{lt_TPOseenMappingThmLem}
Let $q\in(1,\infty)$. Then 
\begin{align*}
&\ALTP:\WSRsigmacompl{2,1}{q}(\grp)\ra\LRsigmacompl{q}(\grp),\quad \ALTP\wvel:= \partial_t\wvel-\Delta{\wvel} -\rey\partial_1{\wvel}
\end{align*}
is a homeomorphism. Moreover $\norm{\ALTPinverse}\leq \Cc{C}\,\polynomial(\rey,\per)$,
where $\Cclast{C}=\Cclast{C}(q)$ and $\polynomial(\rey,\per)$ is a polynomial in $\rey$ and $\per$.
\end{lem}
\begin{proof}
See \cite[Theorem 4.8]{mrtpns}.
\end{proof}

\begin{lem}\label{lt_OseenMappingThmLem}
For $q\in (1,2)$ 
\begin{align*}
\ALOseen:\xoseen{q,r}(\R^3)\ra\LRsigma{q,r}(\R^3),\quad\ALOseen\vvel:=  -\Delta{\vvel} -\rey\partial_1{\vvel}
\end{align*}
a homeomorphism. Moreover 
$\norm{\ALOseeninverse}\leq\Cc[OseenAprioriConst]{C}$ 
with $\const{OseenAprioriConst}$ independent on $\rey$.
\end{lem}
\begin{proof}
See \cite[Theorem VII.4.1]{galdi:book1}.
\end{proof}

\begin{lem}\label{ss_MaxRegThmLem}
If $q\in(1,2)$, $r\in(4,\infty)$ and $\rey\neq 0$, 
then
\begin{align*}%\label{xxxUNUSEDLABELxxxss_MaxRegThmOseenMaxReg}
\begin{aligned}
&\ALTP: \xoseen{q,r}(\R^3)\times\WSRsigmacompl{2,1}{q,r}(\grp)\ra\LRsigma{q,r}(\grp),\\
&\ALTP(\vvel,\wvel):= \partial_t\wvel-\Delta\bp{\vvel+\wvel} -\rey\partial_1\bp{\vvel+\wvel}
\end{aligned}
\end{align*}
is a homeomorphism. Moreover 
\begin{align}\label{ss_MaxRegThmOseenInvBound}
\norm{\ALTPinverse} \leq \Cc[lt_MaxRegThmOseenMaxRegConst]{C}\, \polynomial(\rey,\per),
\end{align}
where $\polynomial(\rey,\per)$ is a polynomial in $\rey$ and $\per$, and $\Cclast{C}=\Cclast{C}(q)$.
\end{lem}
\begin{proof}\newCCtr[c]{ss_MaxRegThm}
Recalling $\LRsigma{q,r}(\grp)=\LRsigma{q,r}(\R^3)\oplus\LRsigmacompl{q,r}(\grp)$ from \eqref{lrsigmaDecomp}, 
Lemma \ref{lt_TPOseenMappingThmLem} and Lemma \ref{lt_OseenMappingThmLem} concludes the proof.
\end{proof}

\begin{lem}\label{ss_PressureMappingLem}
Let $q\in(1,3)$ and $r\in(1,\infty)$. 
Then 
\begin{align*}%\label{xxxUNUSEDLABELxxxss_PressureMappingLemGradmap}
\begin{aligned}
\gradmap: \xpres{q,r}(\grp)\ra \gradspace{q,r}(\grp),\quad \gradmap\,\upres:=\grad\upres
\end{aligned}
\end{align*}
is a homeomorphism.
\end{lem}
\begin{proof}
See \cite[Lemma 5.4]{mrtpns}.
\end{proof}

\begin{proof}[Proof of Theorem \ref{ss_StrongSolThm}]\newCCtr[c]{ss_StrongSolThm}
We can use the Helmholtz projection
to eliminate the pressure term $\grad\upres$ in \eqref{ss_nsongrp}.
More precisely, we shall first study
\begin{align}\label{ss_nssol}
\begin{pdeq}
&\partial_t\uvel -\Delta\uvel -\rey\partial_1\uvel + \hproj\bb{\nsnonlin{\uvel}}= \hproj f && \tin\grp,\\
&\Div\uvel =0 && \tin\grp.
\end{pdeq}
\end{align} 
After solving \eqref{ss_nssol}, a pressure term $\upres$ can be constructed such that $(\uvel,\upres)$ solves \eqref{ss_nsongrp}.

We first show that any pair of vector fields 
$(\vvel,\wvel)\in\xoseen{q,r}(\R^3)\times\WSRsigmacompl{2,1}{q,r}(\grp)$ satisfies 
$\nsnonlin{(\vvel+\wvel)}\in\LR{q}(\grp)^3\cap\LR{r}(\grp)^3$. 
Recalling \eqref{ss_StrongSolThmEmbeddingLemGradvr} and \eqref{ss_StrongSolThmEmbeddingLemvinfty}, we find that 
\begin{align}\label{ss_StrongSolThmNonlintermEst1}
\norm{\nsnonlin{\vvel}}_r \leq \norm{\vvel}_\infty\,\norm{\grad\vvel}_r \leq 
\Cc{ss_StrongSolThm} \norm{\vvel}_{\xoseen{q,r}(\R^3)}^2.
\end{align}
Moreover, employing H\"older's inequality and recalling \eqref{ss_StrongSolThmEmbeddingLemGradv2} we deduce 
\begin{align}\label{xxxUNUSEDLABELxxxss_StrongSolThmNonlintermEst2}
\norm{\nsnonlin{\vvel}}_q \leq \norm{\vvel}_{\frac{2q}{2-q}}\,\norm{\grad\vvel}_2 \leq 
\Cc{ss_StrongSolThm} \norm{\vvel}_{\xoseen{q,r}(\R^3)}^2.
\end{align}
We also observe that
\begin{align}\label{xxxUNUSEDLABELxxxss_StrongSolThmNonlintermEst3}
\norm{\nsnonlinb{\vvel}{\wvel}}_r \leq \norm{\vvel}_\infty\,\norm{\grad\wvel}_r \leq \Cc{ss_StrongSolThm} 
\norm{\vvel}_{\xoseen{q,r}(\R^3)}\,\norm{\wvel}_{2,1q,r}
\end{align}
and 
\begin{align}\label{xxxUNUSEDLABELxxxss_StrongSolThmNonlintermEst4}
\norm{\nsnonlinb{\vvel}{\wvel}}_q \leq \norm{\vvel}_\infty\,\norm{\grad\wvel}_q \leq \Cc{ss_StrongSolThm} 
\norm{\vvel}_{\xoseen{q,r}(\R^3)}\,\norm{\wvel}_{2,1q,r}.
\end{align}
Similarly, recalling \eqref{ss_StrongSolThmEmbeddingLemGradvinfty} we can estimate
\begin{align}\label{xxxUNUSEDLABELxxxss_StrongSolThmNonlintermEst5}
\norm{\nsnonlinb{\wvel}{\vvel}}_r \leq \norm{\wvel}_r\,\norm{\grad\vvel}_\infty \leq \Cc{ss_StrongSolThm} 
\norm{\wvel}_{2,1q,r}\,\norm{\vvel}_{\xoseen{q,r}(\R^3)}
\end{align}
and 
\begin{align}\label{xxxUNUSEDLABELxxxss_StrongSolThmNonlintermEst6}
\norm{\nsnonlinb{\wvel}{\vvel}}_q \leq \norm{\wvel}_q\,\norm{\grad\vvel}_\infty \leq \Cc{ss_StrongSolThm} 
\norm{\wvel}_{2,1q,r}\,\norm{\vvel}_{\xoseen{q,r}(\R^3)}.
\end{align}
By \eqref{ss_StrongSolThmEmbeddingLemwinfty} it follows that also 
\begin{align}\label{xxxUNUSEDLABELxxxss_StrongSolThmNonlintermEst7}
\norm{\nsnonlinb{\wvel}{\wvel}}_r \leq \norm{\wvel}_\infty\,\norm{\grad\wvel}_r \leq \Cc{ss_StrongSolThm} 
\norm{\wvel}_{2,1q,r}^2
\end{align}
and 
\begin{align}\label{ss_StrongSolThmNonlintermEst8}
\norm{\nsnonlinb{\wvel}{\vvel}}_q \leq \norm{\wvel}_q\,\norm{\grad\vvel}_\infty 
\leq \Cc{ss_StrongSolThm} 
\norm{\wvel}_{2,1,q,r}\,\norm{\vvel}_{\xoseen{q,r}(\R^3)}.
\end{align}
Combining \eqref{ss_StrongSolThmNonlintermEst1}--\eqref{ss_StrongSolThmNonlintermEst8}, we conclude $\nsnonlin{(\vvel+\wvel)}\in\LR{q}(\grp)^3\cap\LR{r}(\grp)^3$ and
\begin{align}\label{ss_StrongSolThmNonlintermEstFinal}
\norm{\hproj\bb{\nsnonlin{(\vvel+\wvel)}}}_{\LRsigma{q,r}(\grp)}\leq \Cc[ss_StrongSolThmNonlintermEstFinalConst]{ss_StrongSolThm} 
\norm{(\vvel,\wvel)}_{\xoseen{q,r}(\R^3)\times\WSRsigmacompl{2,1}{q,r}(\grp)}^2.
\end{align}
Recalling Lemma \ref{ss_MaxRegThmLem}, we can now define the map
\begin{align*}
\begin{aligned}
&\fpmap:\xoseen{q,r}(\R^3)\times\WSRsigmacompl{2,1}{q,r}(\grp)\ra\xoseen{q,r}(\R^3)\times\WSRsigmacompl{2,1}{q,r}(\grp),\\
&\fpmap(\vvel,\wvel):=\ALTPinverse\bp{\hproj f-\hproj\bb{\nsnonlin{(\vvel+\wvel)}}}.
\end{aligned}
\end{align*}
By construction, a fixed point $(\vvel,\wvel)$ of $\fpmap$ yields a solution $\uvel:=\vvel+\wvel$ to \eqref{ss_nssol}.  
As $\xoseen{q,r}(\R^3)$ and $\WSRsigmacompl{2,1}{q,r}(\grp)$ are Banach spaces, we shall employ Banach's fixed point theorem to show existence of such a fixed point. To this end, we recall \eqref{ss_MaxRegThmOseenInvBound} and estimate
\begin{align}
\begin{aligned}\label{ss_StrongSolThmEstOfL}
\norm{\fpmap(\vvel,\wvel)} &\leq \const{lt_MaxRegThmOseenMaxRegConst} \polynomial(\rey,\per)\,\bp{
\norm{\hproj f}_{\LRsigma{q,r}(\grp)} + \norm{\hproj\bb{\nsnonlin{(\vvel+\wvel)}}}_{\LRsigma{q,r}(\grp)} } \\
&\leq \Cc[ss_StrongSolThmFixedpointconst1]{ss_StrongSolThm}\polynomial(\rey,\per)\,\bp{ \const{ss_StrongSolThmEps} + 
\norm{(\vvel,\wvel)}_{\xoseen{q,r}(\R^3)\times\WSRsigmacompl{2,1}{q,r}(\grp)}^2\, }. 
\end{aligned}
\end{align}
Consequently, $\fpmap$ is a self-mapping on the ball $\overline{B_\rho}\subset\xoseen{q,r}(\R^3)\times\WSRsigmacompl{2,1}{q,r}(\grp)$ provided 
$\rho$ and $\const{ss_StrongSolThmEps}$ satisfy
\begin{align*}%\label{xxxUNUSEDLABELxxxss_StrongSolThmrhoeq}
\const{ss_StrongSolThmFixedpointconst1}\polynomial(\rey,\per)\bp{\const{ss_StrongSolThmEps} + \rho^2} \leq \rho.
\end{align*}
The above inequality is satisfied if we, for example, choose
\begin{align}\label{ss_StrongSolThmParameterChoice}
\rho:=\frac{1}{4\const{ss_StrongSolThmFixedpointconst1}\polynomial(\rey,\per)},\quad 
\const{ss_StrongSolThmEps}:= \frac{1}{16\const{ss_StrongSolThmFixedpointconst1}^2\polynomial(\rey,\per)^2}.
\end{align}
With this choice of parameters, we further have for $(\vvel_1,\wvel_1),(\vvel_2,\wvel_2)\in\overline{B_\rho}$:
\begin{align*}
\norm{\fpmap(\vvel_1,\wvel_1)-\fpmap(\vvel_2,\wvel_2)} &\leq 
\const{ss_StrongSolThmFixedpointconst1}\polynomial(\rey,\per)\, \norm{(\vvel_1,\wvel_1)-(\vvel_2,\wvel_2)}^2\\
&\leq \const{ss_StrongSolThmFixedpointconst1}\polynomial(\rey,\per)\,2\rho\, \norm{(\vvel_1,\wvel_1)-(\vvel_2,\wvel_2)} \\
&\leq \half\norm{(\vvel_1,\wvel_1)-(\vvel_2,\wvel_2)}. 
\end{align*}
Thus, $\fpmap$ becomes a contractive self-mapping. By 
Banach's fixed point theorem, $\fpmap$ then has a unique fixed point in $\overline{B_\rho}$.  

Finally, we construct the pressure. By \eqref{ss_StrongSolThmNonlintermEstFinal}, $\nsnonlin{\uvel}\in\LR{q}(\grp)^3\cap\LR{r}(\grp)^3$. Recalling
Lemma \ref{ss_PressureMappingLem}, the function 
\begin{align*}
\upres:=\gradmap^{-1}\bp{\bb{\id-\hproj}\bp{f-\nsnonlin{\uvel}}}
\end{align*}
belongs to $\xpres{q,r}(\grp)$. Clearly, $(\uvel,\upres)$ is a solution to \eqref{ss_nsongrp}. 
\end{proof}

\begin{lem}\label{ss_AddRegOfWeakSolLem}
Let $\rey\neq 0$ and $\weakuvel\in\LRloc{1}(\grp)^3$ 
be a generalized solution to \eqref{ss_nsongrp}, that is, it satisfies 
\eqref{UniquenessClassDefDefofweaksol} for all $\Phi\in\CRcisigma(\grp)$, with
$\weakuvel\in\LR{2}\bp{(0,\per);\DSRNsigma{1}{2}(\R^3)}$ and $\projcompl\weakuvel\in\LR{\infty}\bp{(0,\per);\LR{2}(\R^3)^3}$. 
If for some $q\in(1,\frac{5}{4}\big]$
\begin{align}\label{ss_AddRegOfWeakSolLemDatacond1}
f\in\LR{q}(\grp)^3\cap\LR{\frac{3}{2}}(\grp)^3
\end{align} 
then
\begin{align}
&\projcompl\weakuvel\in\WSRsigmacompl{2,1}{q}(\grp).\label{ss_AddRegOfWeakSolLemAddRegProjcompl}
\end{align}
If for some $\tq\in(1,\frac{3}{2}\big]$
\begin{align}\label{ss_AddRegOfWeakSolLemDatacond2}
f\in\LR{\tq}(\grp)^3\cap\LR{\frac{3}{2}}(\grp)^3
\end{align} 
then
\begin{align}
&\proj\weakuvel\in\xoseen{\tq}(\R^3).\label{ss_AddRegOfWeakSolLemAddRegProj}
\end{align}
\end{lem}
\begin{proof}\newCCtr[c]{ss_AddRegOfWeakSolLem}
We first assume \eqref{ss_AddRegOfWeakSolLemDatacond1} for some $q\in(1,\frac{5}{4}\big]$.  
Put $\weakvvel:=\proj\weakuvel$ and $\weakwvel:=\projcompl\weakuvel$. 
By assumption $\weakwvel\in\LR{2}(\grp)^3$ and $\grad\weakwvel\in\LR{2}(\grp)^{3\times 3}$, whence
\begin{align}\label{ss_AddRegOfWeakSolLemSummabilitywgradwlow}
\nsnonlin{\weakwvel}\in\LR{1}(\grp)^3.
\end{align}
Employing first H\"older's inequality and then a Sobolev-type inequality, see for example \cite[Lemma II.2.2]{galdi:book1} invoked with $n=3$, $q=2$ and $r=\frac{10}{3}$, we estimate
\begin{align}\label{ss_AddRegOfWeakSolLemSummabilitywgradwhighpre}
\begin{aligned}
\iper\int_0^\per\int_{\R^3}\snorm{\nsnonlin{\weakwvel}}^{\frac{5}{4}}\,\dx\dt &\leq 
\iper\int_0^\per \norm{\grad\weakwvel(\cdot,t)}_2^{\frac{5}{4}}\, \norm{\weakwvel(\cdot,t)}_{\frac{10}{3}}^{\frac{5}{4}}\,\dt\\
&\leq \Cc{ss_AddRegOfWeakSolLem}\,\iper\int_0^\per \norm{\grad\weakwvel(\cdot,t)}_2^{\frac{5}{4}}\, 
\Bp{\norm{\grad\weakwvel(\cdot,t)}_2^{\frac{3}{5}}\, \norm{\weakwvel(\cdot,t)}_{2}^{\frac{2}{5}}}^{\frac{5}{4}}\,\dt\\
&\leq\Cc{ss_AddRegOfWeakSolLem}\, \bp{\esssup_{t\in(0,\per)}\norm{\weakwvel(\cdot,t)}_2}^{\frac{1}{2}}\cdot 
\iper\int_0^\per \norm{\grad\weakwvel(\cdot,t)}_2^{2}\,\dx <\infty,
\end{aligned}
\end{align}
whence
\begin{align}\label{ss_AddRegOfWeakSolLemSummabilitywgradwhigh}
\nsnonlin{\weakwvel}\in\LR{\frac{5}{4}}(\grp)^3.
\end{align}
We further deduce, by employing first Minkowski's integral inequality, then H\"older's inequality, and finally the Sobolev embedding 
$\DSRNsigma{1}{2}(\R^3)\embeds\LR{6}(\R^3)$, that 
\begin{align*}
\begin{aligned}
\Bp{\int_{\R^3}\bigg|\iper\int_0^\per \nsnonlin{\weakwvel}\,\dt\bigg|^{\frac{3}{2}}\,\dx}^{\frac{3}{2}}&\leq 
\iper\int_0^\per\Bp{\int_{\R^3} \snorm{\nsnonlin{\weakwvel}}^{\frac{3}{2}}\,\dx}^{\frac{3}{2}}\,\dt \\
&\leq \iper\int_0^\per 
\Bp{\int_{\R^3} \snorm{\weakwvel}^{6}\,\dx}^{\frac{1}{6}}
\Bp{\int_{\R^3} \snorm{\grad\weakwvel}^{2}\,\dx}^{\frac{1}{2}}
\,\dt\\
&\leq \iper\int_0^\per \int_{\R^3} \snorm{\grad\weakwvel}^{2}\,\dx\dt<\infty.
\end{aligned}
\end{align*}
Consequently, we have 
\begin{align}\label{ss_AddRegOfWeakSolLemSummabilityprojwgradwhigh}
\proj\bb{\nsnonlin{\weakwvel}}\in\LR{\frac{3}{2}}(\grp)^3.
\end{align} 
Recalling Lemma \ref{lt_projprojcomplinnerprodl1locfunctions}, it is easy to verify from the weak formulation \eqref{UniquenessClassDefDefofweaksol} that  
\begin{align*}%\label{xxxUNUSEDLABELxxxss_AddRegOfWeakSolLemWeakFormulationvvel}
\begin{aligned}
&\int_{\R^3} -\weakvvel\cdot\partial_t\Phi +\grad\weakvvel:\grad\Phi -\rey\partial_1\weakvvel\cdot\Phi
+ \Bp{\nsnonlin{\weakvvel} + \proj\bb{\nsnonlin{\weakwvel}}}\cdot\Phi\,\dx = \int_{\R^3} \proj f\cdot\Phi\,\dx
\end{aligned}
\end{align*}  
for all $\Phi\in\CRcisigma(\R^3)$. This means that $\weakvvel\in\DSRNsigma{1}{2}(\R^3)$ is a generalized solution to the steady-state problem
\begin{align}\label{ss_AddRegOfWeakSolLemEqForvvel}
\begin{pdeq}
&-\Delta\weakvvel -\rey\partial_1\weakvvel + \hproj\bb{\nsnonlin{\weakvvel}}  = \hproj\proj f - \hproj\proj\bb{\nsnonlin{\weakwvel}}  && \tin\R^3,\\
&\Div\weakvvel =0 && \tin\R^3.
\end{pdeq}
\end{align} 
From \eqref{ss_AddRegOfWeakSolLemSummabilitywgradwlow}, \eqref{ss_AddRegOfWeakSolLemSummabilityprojwgradwhigh}, and 
assumption \eqref{ss_AddRegOfWeakSolLemDatacond1}, we deduce
the summability
\begin{align*}%\label{xxxUNUSEDLABELxxxss_AddRegOfWeakSolLemSummabilityRHSeqforvvel}
\proj f - \proj\bb{\nsnonlin{\weakwvel}} \in\LR{q}(\R^3)^3\cap\LR{\frac{3}{2}}(\R^3)^3
\end{align*}
for the right-hand side in \eqref{ss_AddRegOfWeakSolLemEqForvvel}. 
Known results for the steady-state Navier-Stokes problem \eqref{ss_AddRegOfWeakSolLemEqForvvel} then imply
\begin{align}\label{ss_AddRegOfWeakSolLemSummabilityvvel}
\weakvvel\in\xoseen{q}(\R^3)\cap\xoseen{\frac{3}{2}}(\R^3). 
\end{align}
More specifically, we can employ \cite[Lemma X.6.1]{GaldiBookNew}\footnote{Lemma X.6.1 is new in the latest edition of the 
monograph \cite{GaldiBookNew}.} which, although formulated for a three-dimensional exterior domain, also holds for solutions to the
whole-space problem \eqref{ss_AddRegOfWeakSolLemEqForvvel}. By the additional regularity for $\weakvvel$ implied by \eqref{ss_AddRegOfWeakSolLemSummabilityvvel}, it follows that $\grad\weakvvel\in\LR{2}(\R^3)^{3\times 3}$. 
Since by assumption $\weakwvel\in\LR{2}(\grp)^3$, we thus have $\nsnonlinb{\weakwvel}{\weakvvel}\in\LR{1}(\grp)^3$.
In addition, we can deduce as in \eqref{ss_AddRegOfWeakSolLemSummabilitywgradwhighpre} that  
$\nsnonlinb{\weakwvel}{\weakvvel}\in\LR{\frac{5}{4}}(\grp)^3$. Consequently, by interpolation
\begin{align}\label{ss_AddRegOfWeakSolLemSummabilitywgradv}
\nsnonlinb{\weakwvel}{\weakvvel}\in\LR{q}(\grp)^3. 
\end{align}
From \eqref{ss_AddRegOfWeakSolLemSummabilityvvel} we further obtain $\weakvvel\in\LR{\frac{2q}{2-q}}(\R^3)^3$, which combined with
$\grad\weakwvel\in\LR{2}(\grp)^{3\times 3}$ yields
\begin{align}\label{ss_AddRegOfWeakSolLemSummabilityvgradw}
\nsnonlinb{\weakvvel}{\weakwvel}\in\LR{q}(\grp)^3. 
\end{align}
We have now derived enough summability properties for the terms appearing in \eqref{ss_nsongrp} to finalize the proof.
Recalling again Lemma \ref{lt_projprojcomplinnerprodl1locfunctions}, it is easy to verify from the weak formulation \eqref{UniquenessClassDefDefofweaksol} that   
\begin{align}\label{ss_AddRegOfWeakSolLemWeakFormulationwvel}
\begin{aligned}
&\iper\int_0^\per\int_{\R^3} -\weakwvel\cdot\partial_t\Phi +\grad\weakwvel:\grad\Phi -\rey\partial_1\weakwvel\cdot\Phi\\
&\qquad + \Bp{\projcompl[\nsnonlin{\weakwvel}] + \nsnonlinb{\weakwvel}{\weakvvel} + \nsnonlinb{\weakvvel}{\weakwvel}}\cdot\Phi\,\dx\dt  = \iper\int_0^\per\int_{\R^3} \projcompl f\cdot\Phi\,\dx\dt
\end{aligned}
\end{align}  
for all $\Phi\in\CRcisigma(\grp)$. The summability of $\weakwvel$ and $\grad\weakwvel$ together the summability properties 
obtained for $\nsnonlin{\weakwvel}$, $\nsnonlinb{\weakwvel}{\weakvvel}$, and $\nsnonlinb{\weakvvel}{\weakwvel}$ above enables us to 
extend \eqref{ss_AddRegOfWeakSolLemWeakFormulationwvel} to all $\Phi\in\SR(\grp)$. Thus the system 
\begin{align*}%\label{xxxUNUSEDLABELxxxss_AddRegOfWeakSolLemEqForwvel}
\begin{pdeq}
&\partial_t\weakwvel -\Delta\weakwvel -\rey\partial_1\weakwvel = \hproj\projcompl f - \hproj\Bb{\projcompl\bb{\nsnonlin{\weakwvel}} +\nsnonlinb{\weakwvel}{\weakvvel}
+\nsnonlinb{\weakvvel}{\weakwvel} }  && \tin\grp,\\
&\Div\weakwvel =0 && \tin\grp
\end{pdeq}
\end{align*}
is satisfied as an identity in $\TDR(\grp)$. From \eqref{ss_AddRegOfWeakSolLemSummabilitywgradwlow}, \eqref{ss_AddRegOfWeakSolLemSummabilitywgradwhigh},
\eqref{ss_AddRegOfWeakSolLemSummabilitywgradv}, \eqref{ss_AddRegOfWeakSolLemSummabilityvgradw}, and 
the assumptions on $f$, we conclude that
\begin{align*}
\hproj\projcompl f - \hproj\Bb{\projcompl\bb{\nsnonlin{\weakwvel}} +\nsnonlinb{\weakwvel}{\weakvvel}
+\nsnonlinb{\weakvvel}{\weakwvel} } \in\LR{q}(\grp)^3.
\end{align*}
Consequently, Lemma \ref{lt_TPOseenMappingThmLem} combined with Lemma \ref{lt_TPOseenUniqueness} implies \eqref{ss_AddRegOfWeakSolLemAddRegProjcompl}. 

Finally, assume \eqref{ss_AddRegOfWeakSolLemDatacond2} for some $\tq\in(1,\frac{3}{2}\big]$.
In view of \eqref{ss_AddRegOfWeakSolLemSummabilitywgradwlow} and  \eqref{ss_AddRegOfWeakSolLemSummabilityprojwgradwhigh},
we deduce 
\begin{align*}
\proj f - \proj\bb{\nsnonlin{\weakwvel}} \in\LR{\tq}(\R^3)^3\cap\LR{\frac{3}{2}}(\R^3)^3.
\end{align*}
Recalling that $\weakvvel$ solves \eqref{ss_AddRegOfWeakSolLemEqForvvel}, utilizing once more \cite[Lemma X.6.1]{GaldiBookNew}
we conclude \eqref{ss_AddRegOfWeakSolLemAddRegProj}.
\end{proof}

\begin{proof}[Proof of Theorem \ref{ss_EnergyEqThm}]\newCCtr[c]{ss_EnergyEqThm}
The proof relies on the summability properties of the solution $\uvel=\vvel+\wvel$ being sufficient to multiply \eqref{ss_nsongrp}
with $\uvel$ itself and subsequently integrate over space and time. Due to the different summability properties of $\vvel$ and
$\wvel$, it is more convenient to carry out this process for $\vvel$ and $\wvel$ separately. 
Applying first $\projcompl$ and then $\hproj$ to both sides in \eqref{ss_nsongrp}, we obtain
\begin{align}\label{ss_EnergyEqEqForwvel}
\partial_t\wvel -\Delta\wvel -\rey\partial_1\wvel = \hproj\projcompl f - \hproj\Bb{\projcompl\bb{\nsnonlin{\wvel}} +\nsnonlinb{\wvel}{\vvel}
+\nsnonlinb{\vvel}{\wvel} },  
\end{align}
which we multiply with $\wvel$ and integrate $\grp$. We can easily verify that the product of $\wvel$ with each term in \eqref{ss_EnergyEqEqForwvel}
is integrable over $\grp$. For example, we observe that 
\begin{align*}
\iper\int_0^\per\int_{\R^3} \snorml{\partial_t\wvel\cdot\wvel}\,\dx\dt\leq 
\norm{\partial_t\wvel}_{4}\norm{\wvel}_{\frac{4}{3}}\leq 
\norm{\wvel}_{2,1q,r}^2.
\end{align*}
Similarly, one can verify for all the other terms in \eqref{ss_EnergyEqEqForwvel} that the product with $\wvel$ can be integrated 
over $\grp$. We thus conclude that 
\begin{align}\label{ss_EnergyEqTestwithw}
\begin{aligned}
&\iper\int_0^\per\int_{\R^3} \partial_t\wvel\cdot\wvel -\Delta\wvel\cdot\wvel 
-\rey\partial_1\wvel\cdot\wvel \,\dx\dt\\ 
&\qquad = \iper\int_0^\per\int_{\R^3} \projcompl f\cdot \wvel - \projcompl\bb{\nsnonlin{\wvel}}\cdot\wvel 
-(\nsnonlinb{\wvel}{\vvel})\cdot\wvel
-(\nsnonlinb{\vvel}{\wvel})\cdot\wvel\,\dx\dt,  
\end{aligned}
\end{align}
where the Helmholtz projection $\hproj$ can be omitted since $\wvel$ is solenoidal. Since $\wvel=\projcompl\wvel$ we can, recalling 
\eqref{lt_projprojcomplinnerprodl1locfunctionsinnerprod}, also 
omit the projection $\projcompl$ in the first two terms on the right-hand side. Moreover, the summability properties of $\wvel$
are sufficient to integrate by parts in each term above. Consequently, we obtain
\begin{align}\label{ss_EnergyEqTestwithwandelimination}
\begin{aligned}
\iper\int_0^\per\int_{\R^3} \grad\wvel:\grad\wvel \,\dx\dt = \iper\int_0^\per\int_{\R^3} f\cdot \wvel -  
(\nsnonlinb{\wvel}{\vvel})\cdot\wvel \,\dx\dt.  
\end{aligned}
\end{align}
We now repeat the procedure with $\vvel$ in the role of $\wvel$. 
Applying first $\proj$ and then $\hproj$ to both sides in \eqref{ss_nsongrp}, we obtain 
\begin{align}\label{ss_EnergyEqEqForv}
-\Delta\vvel -\rey\partial_1\vvel = \hproj\proj f - \hproj\Bb{\proj\bb{\nsnonlin{\wvel}} +\nsnonlin{\vvel} },  
\end{align} 
which we multiply with $\vvel$ and integrate over $\R^3$. Again it should be verified that the product of the terms in \eqref{ss_EnergyEqEqForv}
with $\vvel$ is integrable over $\R^3$. This, however, is standard to show. For example, in view of \eqref{ss_StrongSolThmEmbeddingLemvinfty}
and the fact that 
$\frac{2q}{2-q}\leq \frac{q}{q-1}$ 
it follows that 
\begin{align*}
\snorml{\int_{\R^3}\Delta\vvel\cdot\vvel\,\dx}\leq \norm{\Delta\vvel}_q\norm{\vvel}_{\frac{q}{q-1}}
\leq \norm{\vvel}_{\xoseen{q,r}(\R^3)}^2.
\end{align*}
Similarly, one can verify for all the other terms in \eqref{ss_EnergyEqEqForv} that the product with $\vvel$ can be integrated 
over $\R^3$. We thus conclude that 
\begin{align}\label{ss_EnergyEqTestwithv}
\begin{aligned}
&\int_{\R^3} -\Delta\vvel\cdot\vvel -\rey\partial_1\vvel\cdot\vvel\,\dx =
\int_{\R^3} f\cdot \vvel - (\nsnonlin{\wvel})\cdot\vvel - (\nsnonlin{\vvel})\cdot\vvel \,\dx.  
\end{aligned}
\end{align} 
One may also verify that the summability properties of $\vvel$ are sufficient to integrate by parts in \eqref{ss_EnergyEqTestwithv}. We thereby obtain
\begin{align}\label{ss_EnergyEqTestwithvAfterElimination}
\begin{aligned}
&\int_{\R^3} \grad\vvel:\grad\vvel\dx =
\int_{\R^3} f\cdot \vvel + (\nsnonlinb{\wvel}{\vvel})\cdot\wvel \,\dx.  
\end{aligned}
\end{align} 
Adding together \eqref{ss_EnergyEqTestwithwandelimination} and \eqref{ss_EnergyEqTestwithvAfterElimination} we deduce
\begin{align*}%\label{xxxUNUSEDLABELxxxss_EnergyEqAfterAddition}
\begin{aligned}
\iper\int_0^\per\int_{\R^3} \snorm{\grad\wvel}^2+\snorm{\grad\vvel}^2\,\dx\dt = \iper\int_0^\per\int_{\R^3} f\cdot (\vvel+\wvel) \,\dx\dt.  
\end{aligned}
\end{align*} 
Since 
\begin{align*}
\iper\int_0^\per\int_{\R^3} \grad\vvel:\grad\wvel\,\dx\dt = 0,
\end{align*}
we finally conclude \eqref{EnergyEqEE}.
\end{proof}

\begin{proof}[Proof of Theorem \ref{ss_UniquenessThm}]\newCCtr[c]{ss_UniquenessThm}
Choosing $\const{ss_UniquenessThmConst}\leq\const{ss_StrongSolThmEps}$, we obtain by Theorem \ref{ss_StrongSolThm} a solution $(\uvel,\upres)$ in the class 
\eqref{ss_StrongSolThmSolSpace} (with $\uvel=\vvel+\wvel$). From the proof of Theorem \ref{ss_StrongSolThm}, 
in particular \eqref{ss_StrongSolThmEstOfL}--\eqref{ss_StrongSolThmParameterChoice}, we recall that $\uvel\in\overline{\B_{\rho}}\subset\xoseen{q,r}(\R^3)\times\WSRsigmacompl{2,1}{q,r}(\grp)$ 
with $\rho:=\const{ss_UniquenessThmConst}^\half$,
which means that 
\begin{align}\label{ss_UniquenessThmBoundOnSol}
\norm{(\vvel,\wvel)}_{\xoseen{q,r}(\R^3)\times\WSRsigmacompl{2,1}{q,r}(\grp)} \leq \const{ss_UniquenessThmConst}^\half.
\end{align}
Now recall Definition \ref{ss_UniquenessClassDef} and consider a \emph{physically reasonable weak solution} $\weakuvel$ corresponding to the same data $f$.
Put $\weakvvel:=\proj\weakuvel$ and $\weakwvel:=\projcompl\weakuvel$.
We shall verify that the regularity of $\weakvvel$ and $\weakwvel$ ensured by Lemma \ref{ss_AddRegOfWeakSolLem} enables us to use
$\uvel=\vvel+\wvel$ as a test function in the weak formulation for $\weakuvel=\weakvvel+\weakwvel$. Observe for example that \eqref{ss_AddRegOfWeakSolLemAddRegProj}
implies $\weakvvel\in\LR{\frac{2q}{2-q}}(\R^3)^3$, from which it follows, since 
the H\"older conjugate $\big(\frac{2q}{2-q}\big)'=\frac{2q}{3q-2}$ belongs to the interval $(q,r)$, that  
\begin{align}\label{ss_UniquenessThmAddSummability1}
\weakvvel\cdot\partial_t\wvel\in\LR{1}(\grp).
\end{align}
Moreover, since by assumption $\weakwvel\in\LR{2}(\grp)^3$, we also have
\begin{align}\label{xxxUNUSEDLABELxxxss_UniquenessThmAddSummability2}
\weakwvel\cdot\partial_t\wvel\in\LR{1}(\grp).
\end{align}
In a similar manner, one may verify that 
\begin{align}\label{ss_UniquenessThmAddSummability3}
\grad\weakvvel:\grad\vvel,\ 
\grad\weakvvel:\grad\wvel,\
\grad\weakwvel:\grad\vvel,\ 
\grad\weakwvel:\grad\wvel\in\LR{1}(\grp)^3.
\end{align}
From \eqref{ss_AddRegOfWeakSolLemAddRegProj} and the initial regularity of $\weakvvel$, we obtain 
$\partial_1\weakvvel\in\LR{q}(\R^3)^3\cap\LR{2}(\R^3)^3$. 
Thus, since $\vvel\in\LR{\frac{2q}{2-q}}(\R^3)^3$ and the H\"older conjugate $\big(\frac{2q}{2-q}\big)'=\frac{2q}{3q-2}$ belongs to the interval $(q,2)$, we deduce
\begin{align}\label{xxxUNUSEDLABELxxxss_UniquenessThmAddSummability4}
\partial_1\weakvvel\cdot\vvel\in\LR{1}(\R^3)^3.
\end{align}
In view of \eqref{ss_AddRegOfWeakSolLemAddRegProjcompl}, the same argument yields
\begin{align}\label{xxxUNUSEDLABELxxxss_UniquenessThmAddSummability5}
\partial_1\weakwvel\cdot\vvel\in\LR{1}(\grp)^3.
\end{align}
It is easy to see that
\begin{align}\label{xxxUNUSEDLABELxxxss_UniquenessThmAddSummability6}
\partial_1\weakvvel\cdot\wvel,\ \partial_1\weakwvel\cdot\wvel\in\LR{1}(\grp)^3.
\end{align}
By Lemma \ref{ss_StrongSolThmEmbeddingLem}, we have 
$\vvel\in\LR{\frac{2q}{2-q}}(\R^3)^3\cap\LR{\infty}(\R^3)^3$. 
Moreover, recalling the embedding 
$\DSRNsigma{1}{2}(\R^3)\embeds\LR{6}(\R^3)$, we find that $\weakvvel\in\LR{\frac{2q}{2-q}}(\R^3)^3\cap\LR{6}(\R^3)^3$.
We thus see that $\vvel,\wvel,\weakvvel\in\LR{4}(\grp)^3$, from which one can deduce that 
\begin{align}\label{xxxUNUSEDLABELxxxss_UniquenessThmAddSummability7}
(\nsnonlin{\weakvvel})\cdot\vvel,\
(\nsnonlin{\weakvvel})\cdot\wvel,\
(\nsnonlinb{\weakvvel}{\weakwvel})\cdot\vvel,\
(\nsnonlinb{\weakvvel}{\weakwvel})\cdot\wvel\in\LR{1}(\grp)^3.
\end{align}
Lemma \ref{ss_StrongSolThmEmbeddingLem} also yields $\wvel\in\LR{\infty}(\grp)^3$, whence
\begin{align}\label{ss_UniquenessThmAddSummability8}
(\nsnonlinb{\weakwvel}{\weakvvel})\cdot\vvel,\
(\nsnonlinb{\weakwvel}{\weakvvel})\cdot\wvel,\
(\nsnonlinb{\weakwvel}{\weakwvel})\cdot\vvel,\
(\nsnonlinb{\weakwvel}{\weakwvel})\cdot\wvel\in\LR{1}(\grp)^3.
\end{align}
Finally, recalling that $\big(\frac{2q}{2-q}\big)'=\frac{2q}{3q-2}\in(q,2)$, the summability of $f$ implies
\begin{align}\label{ss_UniquenessThmAddSummability9}
f\cdot\vvel,\ f\cdot\wvel\in\LR{1}(\grp)^3.
\end{align}
From the summability properties \eqref{ss_UniquenessThmAddSummability1}--\eqref{ss_UniquenessThmAddSummability9}, we 
conclude, by a standard approximation argument, that 
$\uvel=\vvel+\wvel$ can be used as a test function in the weak formulation for $\weakuvel=\weakvvel+\weakwvel$ and thus obtain
\begin{align}\label{ss_UniquenessThmTestOfWeakEqWithStrongSol}
\begin{aligned}
\iper\int_0^\per\int_{\R^3} -\weakwvel\cdot\partial_t\wvel +\grad\weakuvel:\grad\uvel -\rey\partial_1\weakuvel\cdot\uvel + (\nsnonlin{\weakuvel})\cdot\uvel\,\dx\dt  = \iper\int_0^\per\int_{\R^3} f\cdot\uvel\,\dx\dt.
\end{aligned}
\end{align}  
We now consider the equation 
\begin{align}\label{ss_UniquenessThmEqforwvel}
\partial_t\uvel -\Delta\uvel -\rey\partial_1\uvel + \grad\upres + \nsnonlin{\uvel}= f \quad\tin\grp
\end{align}
satisfied by the strong solution. We shall multiply \eqref{ss_UniquenessThmEqforwvel} with $\weakuvel$ and integrate over $\grp$.
With the aid of Lemma \ref{ss_AddRegOfWeakSolLem} and Lemma \ref{ss_StrongSolThmEmbeddingLem}, one can verify as above that 
the resulting integral is well-defined. We thus obtain
\begin{align*}
\begin{aligned}
\iper\int_0^\per\int_{\R^3} 
\partial_t\wvel\cdot\weakwvel -\Delta\uvel\cdot\weakuvel -\rey\partial_1\uvel\cdot\weakuvel + \grad\upres\cdot\weakuvel 
+ (\nsnonlin{\uvel})\cdot\weakuvel
\,\dx\dt  = \iper\int_0^\per\int_{\R^3} f\cdot\weakuvel\,\dx\dt.
\end{aligned}
\end{align*}  
Recalling
\eqref{ss_UniquenessThmAddSummability3}--\eqref{ss_UniquenessThmAddSummability8}, we see that the following integration by parts is valid
\begin{align}\label{ss_UniquenessThmTestOfStrongEqWithWeakSol}
\begin{aligned}
&\iper\int_0^\per\int_{\R^3} 
\partial_t\wvel\cdot\weakwvel +\grad\uvel:\grad\weakuvel +\rey\uvel\cdot\partial_1\weakuvel  
- (\nsnonlinb{\uvel}{\weakuvel})\cdot\uvel
\,\dx\dt =\iper\int_0^\per\int_{\R^3} f\cdot\weakuvel\,\dx\dt.
\end{aligned}
\end{align}  
Adding together \eqref{ss_UniquenessThmTestOfWeakEqWithStrongSol} and \eqref{ss_UniquenessThmTestOfStrongEqWithWeakSol}, we deduce 
\begin{align}\label{ss_UniquenessThmAfterAddition}
\begin{aligned}
2\,\iper\int_0^\per\int_{\R^3} \grad\Uvel:\grad\uvel\,\dx\dt &= \iper\int_0^\per\int_{\R^3} f\cdot\weakuvel\,\dx\dt + \iper\int_0^\per\int_{\R^3} f\cdot\uvel\,\dx\dt \\ 
&\quad +\iper\int_0^\per\int_{\R^3} \bp{\nsnonlinb{(\uvel-\weakuvel)}{\weakuvel}}\cdot\uvel \,\dx\dt. 
\end{aligned}
\end{align}
Since
\begin{align*}
\begin{aligned}
\iper\int_0^\per\int_{\R^3} \snorml{\grad\weakuvel-\grad\uvel}^2\,\dx\dt =
\iper\int_0^\per\int_{\R^3} \snorml{\grad\weakuvel}^2+\snorml{\grad\uvel}^2 \,\dx\dt -2\,\iper\int_0^\per\int_{\R^3} \grad\Uvel:\grad\uvel\,\dx\dt,
\end{aligned}
\end{align*}
we can utilize \eqref{ss_UniquenessThmAfterAddition} in combination with the energy equality \eqref{EnergyEqEE} satisfied by 
$\uvel$ due to Theorem \ref{ss_EnergyEqThm} and the energy inequality \eqref{UniquenessClassDefEnergyIneq} satisfied by $\weakuvel$ to deduce
\begin{align}\label{ss_UniquenessThmFinalInEq1}
\begin{aligned}
\iper\int_0^\per\int_{\R^3} \snorml{\grad\weakuvel-\grad\uvel}^2\,\dx\dt \leq
\iper\int_0^\per\int_{\R^3} \bp{\nsnonlinb{(\weakuvel-\uvel)}{\weakuvel}}\cdot\uvel \,\dx\dt.
\end{aligned}
\end{align}
Recalling \eqref{ss_StrongSolThmEmbeddingLemGradv2}, we see that $\grad\uvel\in\LR{2}(\grp)^{3\times 3}$.
We already observed that $\uvel,\weakvvel\in\LR{4}(\grp)^3$. 
Thus, an integration by parts yields  
\begin{align}\label{ss_UniquenessThmNonlinReformulation1}
\iper\int_0^\per\int_{\R^3} \bp{\nsnonlinb{\weakvvel}{\uvel}}\cdot\uvel \,\dx\dt = 0.
\end{align}
Since $\weakwvel\in\LR{2}(\grp)^3$ and $\uvel\in\LR{\infty}(\grp)$, it further follows that 
\begin{align}\label{ss_UniquenessThmNonlinReformulation2}
\iper\int_0^\per\int_{\R^3} \bp{\nsnonlinb{\weakwvel}{\uvel}}\cdot\uvel \,\dx\dt = 0.
\end{align}
Adding together \eqref{ss_UniquenessThmNonlinReformulation1} and \eqref{ss_UniquenessThmNonlinReformulation2} we obtain
\begin{align}\label{ss_UniquenessThmNonlinReformulation}
\iper\int_0^\per\int_{\R^3} \bp{\nsnonlinb{\weakuvel}{\uvel}}\cdot\uvel \,\dx\dt = 0.
\end{align}
Consequently, we can rewrite \eqref{ss_UniquenessThmFinalInEq1} as 
\begin{align}\label{ss_UniquenessThmFinalInEq2}
\begin{aligned}
\iper\int_0^\per\int_{\R^3} \snorml{\grad\weakuvel-\grad\uvel}^2\,\dx\dt \leq
\iper\int_0^\per\int_{\R^3} \bp{\nsnonlinb{(\weakuvel-\uvel)}{(\weakuvel-\uvel)}}\cdot\uvel \,\dx\dt.
\end{aligned}
\end{align}
Recalling the embedding 
$\DSRNsigma{1}{2}(\R^3)\embeds\LR{6}(\R^3)$, we estimate
\begin{align}\label{ss_UniquenessThmFinalInEqEst}
\begin{aligned}
&\snormL{\iper\int_0^\per\int_{\R^3} \bp{\nsnonlinb{(\weakuvel-\uvel)}{(\weakuvel-\uvel)}}\cdot\uvel \,\dx\dt}\\
&\qquad \leq \iper\int_0^\per 
\norm{\weakuvel(\cdot,t)-\uvel(\cdot,t)}_6\,
\norm{\grad\weakuvel(\cdot,t)-\grad\uvel(\cdot,t)}_2\, 
\norm{\uvel(\cdot,t)}_3\,\dt\\
&\qquad \leq \esssup_{t\in(0,\per)}\norm{\uvel(\cdot,t)}_3\, \iper\int_0^\per 
\norm{\grad\weakuvel(\cdot,t)-\grad\uvel(\cdot,t)}_2^2\,\dt.
\end{aligned}
\end{align}
Now we finally need the assumption $q\leq \frac{6}{5}$, which implies $\frac{2q}{2-q}\leq 3$. Consequently, from the fact that 
$\norm{\vvel}_{\frac{2q}{2-q}}\leq \norm{\vvel}_{\xoseen{q,r}(\R^3)}$ and, by Lemma \ref{ss_StrongSolThmEmbeddingLem}, 
$\norm{\vvel}_{\infty}\leq \const{ss_StrongSolThmEmbeddingLemConstvinfty}\norm{\vvel}_{\xoseen{q,r}(\R^3)}$, we obtain
\begin{align}\label{ss_UniquenessThmvvelLR3est}
\norm{\vvel}_{\LR{3}(\R^3))} \leq \Cc{ss_UniquenessThm} 
\norm{\vvel}_{\xoseen{q,r}(\R^3)}.
\end{align}
Since $\wvel\in\WSRsigmacompl{2,1}{r,q}(\grp)\embeds\WSR{1}{3}(\grp)^3\embeds\WSR{1}{3}\bp{(0,\per);\LR{3}(\R^3)^3}$, standard Sobolev embedding yields 
$\wvel\in\LR{\infty}\bp{(0,\per);\LR{3}(\R^3)^3}$ with 
\begin{align}\label{ss_UniquenessThmwvelLR3est}
\norm{\wvel}_{\LR{\infty}((0,\per);\LR{3}(\R^3))} \leq \Cc{ss_UniquenessThm} \norm{\wvel}_{2,1,q,r}.
\end{align}
Combining \eqref{ss_UniquenessThmvvelLR3est} and \eqref{ss_UniquenessThmwvelLR3est}, we obtain
\begin{align*}
\norm{\uvel}_{\LR{\infty}((0,\per);\LR{3}(\R^3))} \leq \Cc[ss_UniquenessThmuveLinftyL3Const]{ss_UniquenessThm} 
\norm{(\vvel,\wvel)}_{\xoseen{q,r}(\R^3)\times\WSRsigmacompl{2,1}{q,r}(\grp)}.
\end{align*}
This estimate together with \eqref{ss_UniquenessThmBoundOnSol}, \eqref{ss_UniquenessThmFinalInEq2} and \eqref{ss_UniquenessThmFinalInEqEst} finally yields
\begin{align*}
\begin{aligned}
\iper\int_0^\per\int_{\R^3} \snorml{\grad\weakuvel-\grad\uvel}^2\,\dx\dt \leq 
\const{ss_UniquenessThmuveLinftyL3Const}\, \const{ss_UniquenessThmConst}^\half\,
\iper\int_0^\per\int_{\R^3} \snorml{\grad\weakuvel-\grad\uvel}^2\,\dx\dt.
\end{aligned}
\end{align*}
We conclude that $\weakuvel=\uvel$ if $\const{ss_UniquenessThmConst} < \const{ss_UniquenessThmuveLinftyL3Const}^{-\half}$.
\end{proof}

\begin{rem}
The proof of Theorem \ref{ss_UniquenessThm} follows an idea introduced by \textsc{Galdi} in \cite{Galdi1993_LRPR}. The same method was also used in \cite{Silvestre_TPFiniteKineticEnergy12}
to show a uniqueness result for the time-periodic Navier-Stokes problem in the case $\rey=0$.  
\end{rem}

\begin{proof}[Proof of Theorem \ref{ss_RegularityThm}]\newCCtr[c]{ss_RegularityThm}
By assumption, $(\uvel,\upres)$ is a solution to \eqref{ss_nsongrp} in the class \eqref{ss_StrongSolThmSolSpace} with $\uvel=\vvel+\wvel$.
Applying first $\projcompl$ and then $\hproj$ on both sides in \eqref{ss_nsongrp} we obtain 
\begin{align}\label{ss_RegularityThmEqForwvel}
\begin{pdeq}
&\partial_t\wvel -\Delta\wvel -\rey\partial_1\wvel = \hproj\projcompl f - \hproj\Bb{\projcompl\bb{\nsnonlin{\wvel}} +\nsnonlinb{\wvel}{\vvel}
+\nsnonlinb{\vvel}{\wvel} }  && \tin\grp,\\
&\Div\wvel =0 && \tin\grp.
\end{pdeq}
\end{align} 
We shall ``take half a derivative in time'' on both sides of $\eqref{ss_RegularityThmEqForwvel}$.
We therefore introduce the pseudo-differential operator
\begin{align*}
\partial_t^\half:\SR(\grp)\ra\SR(\grp),\quad \partial_t^\half\psi:=\iFT_\grp\Bb{\big(i\perf k\big)^\half\ft{\psi}},
\end{align*}
which, by duality, extends to an operator 
$\partial_t^\half:\TDR(\grp)\ra\TDR(\grp)$.
Note that $\nsnonlin{\wvel}=\Div\wvel\otimes\wvel$ for a solenoidal vector field $\wvel$. We thus find that
\begin{align}\label{ss_RegularityThm_pthalfmultiplierrep}
\begin{aligned}
\partial_t^\half\Bb{\projcompl\bb{\nsnonlin{\wvel}}}_j &= 
\iFT_\grp\Bb{\frac{\big(1-\projsymbol(\xi,k)\big)\big(i\perf k\big)^\half (i\xi_l)}{\snorm{\xi}^2+i\perf k}\,(\snorm{\xi}^2+i\perf k)\,\ft{\wvel_j\wvel_l}}\\
&=\iFT_\grp\Bb{\Mmultiplier_l(\xi,k)\,\FT_\grp\bb{(\partial_t-\Delta)[\wvel_j\wvel_l]}} 
\end{aligned}
\end{align}
with
\begin{align*}
\Mmultiplier_l:\dualgrp\ra\CNumbers,\quad \Mmultiplier_l(\xi,k):=\frac{\big(1-\projsymbol(\xi,k)\big)\big(i\perf k\big)^\half (i\xi_l)}{\snorm{\xi}^2+i\perf k}.
\end{align*}
Observe that the only zero of the polynomial denumerator of $\Mmultiplier_l$ is $(\xi,k)=(0,0)$. When $k=0$, however, the numerator vanishes
due to the term $\big(1-\projsymbol(\xi,k)\big)$. 
Consequently, we see that $\Mmultiplier_l\in\CRi(\dualgrp)$ and that $\Mmultiplier_l$ is bounded. 
Using the same idea introduced in \cite{mrtpns} based on the transference principle of Fourier multipliers, 
we will show that $\Mmultiplier_l$ is an $\LR{p}(\grp)$-multiplier for all $p\in(1,\infty)$. 
For this purpose, we let $\chi$ be a ``cut-off'' function with
\begin{align*}
\chi\in\CRci(\R;\R),\quad \chi(\eta)=1\ \text{for}\ \snorm{\eta}\leq\half,\quad \chi(\eta)=0\ \text{for}\ \snorm{\eta}\geq 1,
\end{align*}
and define
\begin{align*}
\mmultiplier_l:\R^3\times\R\ra\CNumbers,\quad \mmultiplier_l(\xi,\eta):=\frac{\big(1-\chi(\iperf\eta)\big)\big(i \eta\big)^\half \xi_l}{\snorm{\xi}^2+i \eta}.
\end{align*}
Observe that the numerator in the definition of $\mmultiplier_l$ vanishes in a neighborhood of the only zero $(\xi,\eta)=(0,0)$ of the denumerator.
Away from $(0,0)$, $\mmultiplier_l$ is a rational function with non-vanishing denumerator. Consequently, $\mmultiplier_l$ is smooth and bounded. 
Moreover, as one readily verifies, $\mmultiplier_l$ satisfies 
\begin{align*}%\label{xxxUNUSEDLABELxxxlt_TPOseenMappingThmMarcinkiewczCond}
\sup_{\epsilon\in\set{0,1}^{4}}\sup_{(\xi,\eta)\in\R^3\times\R} \snorml{\xi_1^{\epsilon_1}\xi_2^{\epsilon_2}\xi_3^{\epsilon_3}\eta^{\epsilon_{4}}
\partial_{1}^{\epsilon_1}\partial_{2}^{\epsilon_2}\partial_{3}^{\epsilon_3}\partial_\eta^{\epsilon_{4}}
\mmultiplier_l(\xi,\eta)} < \infty.
\end{align*}
This means that $\mmultiplier_l$ satisfies the condition of Marcinkiewicz's multiplier theorem; see for example \cite[Corollary 5.2.5]{Grafakos1} or \cite[Chapter IV, \S 6]{Stein70}. Consequently, $\mmultiplier_l$ is an $\LR{p}(\R^3\times\R)$-multiplier. Next, we introduce  
$\Phi:\dualgrp\ra\R^3\times\R$, $\Phi(\xi,k):=(\xi,\perf k)$. Clearly, $\Phi$ is a continuous homomorphism between the topological groups $\grp$ and $\R^3\times\R$, the latter being considered a topological group
in the canonical way. Observe that $\Mmultiplier_l=\mmultiplier_l\circ\Phi$.
Since $\mmultiplier_l$ is an $\LR{p}(\R^3\times\R)$-multiplier, it follows from the transference principle of Fourier multipliers on 
groups\footnote{Originally, \textsc{de Leeuw} \cite{Leeuw1965} established the transference principle between the torus group and $\R$. 
\textsc{Edwards} and \textsc{Gaudry} \cite[Theorem B.2.1]{EdwardsGaudryBook} generalized the theorem of \textsc{de Leeuw} to arbitrary locally compact abelian 
groups. We employ this general version with groups $\R^3\times\R$ and $\grp:=\R^3\times\R/\per\Z$. A proof
of the theorem for this particular choice of groups can be found in \cite[Theorem 3.4.5]{habil}.},
see \cite[Theorem B.2.1]{EdwardsGaudryBook} or \cite[Theorem 3.4.5]{habil}, that  
$\Mmultiplier_l$ is an $\LR{p}(\grp)$-multiplier.
Recalling \eqref{ss_RegularityThm_pthalfmultiplierrep}, we thus obtain 
\begin{align}\label{ss_RegularityThmpartialthalfEst}
\forall p\in(1,\infty):\quad \normL{\partial_t^\half\Bb{\projcompl\bb{\nsnonlin{\wvel}}}}_p \leq \Cc{ss_RegularityThm}\,\norm{(\partial_t-\Delta)[\wvel\otimes\wvel]}_p.
\end{align}
Due to $\wvel\in\WSRsigmacompl{2,1}{q,r}(\grp)$ and the fact that, by \eqref{ss_StrongSolThmEmbeddingLemwinfty}, $\wvel\in\LR{\infty}(\grp)$, we have
$\partial_t\wvel_j\wvel_l\in\LR{q}(\grp)\cap\LR{r}(\grp)$ and $\Delta\wvel_j\wvel_l\in\LR{q}(\grp)\cap\LR{r}(\grp)$.
Moreover, since $\frac{r}{2}>q$ we observe that 
\begin{align}\label{ss_RegularityThmgradgradregularity1}
&\grad\wvel_j\cdot\grad\wvel_l\in\LR{q}(\grp)\cap\LR{\frac{r}{2}}(\grp).
\end{align}
Computing
\begin{align*}
(\partial_t-\Delta)[\wvel_j\wvel_l] = 
\partial_t\wvel_j\wvel_l + \wvel_j\partial_t\wvel_l -(\Delta\wvel_j\wvel_l + \wvel_j\Delta\wvel_l + 2 \grad\wvel_j\cdot\grad\wvel_l),
\end{align*}
we conclude by \eqref{ss_RegularityThmpartialthalfEst} that
\begin{align}\label{ss_RegularityThmgradpartialthalfnonlintermregularity1}
\partial_t^\half\Bb{\projcompl\np{\nsnonlin{\wvel}}}\in\LR{q}(\grp)\cap\LR{\frac{r}{2}}(\grp).
\end{align}
We now recall \eqref{ss_StrongSolThmEmbeddingLemGradvinfty}, \eqref{ss_StrongSolThmEmbeddingLemvinfty}, and \eqref{ss_StrongSolThmEmbeddingLemwinfty} to deduce 
\begin{align*}
\begin{aligned}
&\partial_t\wvel_j\vvel_l,\, 
\Delta\wvel_j\vvel_l,\,
\Delta\vvel_j\wvel_l,\,
\grad\wvel_j\cdot\grad\vvel_l\in\LR{q}(\grp)\cap\LR{r}(\grp).
\end{aligned}
\end{align*}
By the same argument as above, we obtain
\begin{align*}
\norm{\partial_t^\half\bb{\nsnonlinb{\wvel}{\vvel}+\nsnonlinb{\vvel}{\wvel}}}_{q} \in\LR{q}(\grp)\cap\LR{r}(\grp)\subset\LR{q}(\grp)\cap\LR{\frac{r}{2}}(\grp).
\end{align*}
We now apply $\partial_t^\half$ to both sides in \eqref{ss_RegularityThmEqForwvel}. Clearly, all differential operators commute with 
$\partial_t^\half$. Recalling definition \eqref{lt_HelmholtzProjDefDef} of the Helmholtz projection in terms of a Fourier multiplier, we also see that 
$\partial_t^\half$ commutes with $\hproj$. Similarly, $\partial_t^\half$ commutes with $\projcompl$. 
Consequently,
after applying $\partial_t^\half$ to both sides in \eqref{ss_RegularityThmEqForwvel}, we obtain 
\begin{align*}
\partial_t\bb{\partial_t^\half\wvel} -\Delta\bb{\partial_t^\half\wvel} -\rey\partial_1\bb{\partial_t^\half\wvel} 
\in\LRsigmacompl{q}(\grp)\cap\LRsigmacompl{\frac{r}{2}}(\grp).
\end{align*} 
Combining now Lemma \ref{lt_TPOseenMappingThmLem} and Lemma \ref{lt_TPOseenUniqueness}, we conclude 
\begin{align}\label{ss_RegularityThmPartailthalfRegularity}
\partial_t^\half\wvel\in\WSRsigmacompl{2,1}{q}(\grp)\cap \WSRsigmacompl{2,1}{\frac{r}{2}}(\grp).
\end{align}
Since 
\begin{align*}
\begin{aligned}
\partial_t\partial_j\wvel &= 
\iFT_\grp\Bb{\frac{\big(1-\projsymbol(\xi,k)\big)\big(i\perf k\big)^\half (i\xi_j)}{\snorm{\xi}^2+i\perf k}\,
\,\FT_\grp\bb{(\partial_t-\Delta)\partial_t^\half\wvel_j}}, 
\end{aligned}
\end{align*}
we deduce by analyzing the multiplier 
\begin{align*}
(\xi,k)\ra\frac{\big(1-\projsymbol(\xi,k)\big)\big(i\perf k\big)^\half (i\xi_j)}{\snorm{\xi}^2+i\perf k}
\end{align*}
in same way as we analyzed $\Mmultiplier_l$ that 
\begin{align*}
\forall p\in(1,\infty):\quad \norm{\partial_t\partial_j\wvel}_p \leq \Cc{ss_RegularityThm} \norm{(\partial_t-\Delta)\partial_t^\half\wvel_j}_p.
\end{align*}
In view of \eqref{ss_RegularityThmPartailthalfRegularity}, we thus have
\begin{align}\label{ss_RegularityThmpartialtpartialjregularity1}
\partial_t\partial_j\wvel \in\LR{q}(\grp)\cap\LR{\frac{r}{2}}(\grp).
\end{align}
Combined with the fact that $\wvel\in\WSR{2,1}{q,r}(\grp)$, it follows that $\grad\wvel\in\WSR{1}{\frac{r}{2}}(\grp)$. Since $\frac{r}{2}>4$, classical Sobolev embedding yields $\WSR{1}{\frac{r}{2}}(\grp)\embeds\LR{\infty}(\grp)$. Thus 
\begin{align}\label{ss_RegularityThmwvelinfty}
\grad\wvel\in\LR{\infty}(\grp).
\end{align}
With this information, we return to \eqref{ss_RegularityThmgradgradregularity1} and conclude that in fact 
\begin{align*}%\label{xxxUNUSEDLABELxxxss_RegularityThmgradgradregularity2}
&\grad\wvel_h\cdot\grad\wvel_m\in\LR{q}(\grp)\cap\LR{r}(\grp).
\end{align*}
We therefore obtain improved regularity in \eqref{ss_RegularityThmgradpartialthalfnonlintermregularity1}, namely
\begin{align*}%\label{xxxUNUSEDLABELxxxss_RegularityThmgradpartialthalfnonlintermregularity2}
\partial_t^\half\bb{\projcompl\np{\nsnonlin{\wvel}}}\in\LR{q}(\grp)\cap\LR{r}(\grp).
\end{align*}
Repeating the argument leading up to \eqref{ss_RegularityThmpartialtpartialjregularity1}, we then deduce
\begin{align}\label{ss_RegularityThmwpartialtpartialjwvel}
\partial_t\partial_j\wvel \in\LR{q}(\grp)\cap\LR{r}(\grp).
\end{align}
We shall now take a full derivative in time on both sides in \eqref{ss_RegularityThmEqForwvel}. Concerning the terms that will then appear on the right-hand side, we observe, recalling 
\eqref{ss_StrongSolThmEmbeddingLemGradvinfty}, 
\eqref{ss_StrongSolThmEmbeddingLemvinfty},
\eqref{ss_StrongSolThmEmbeddingLemwinfty},
\eqref{ss_RegularityThmwvelinfty}, 
and \eqref{ss_RegularityThmwpartialtpartialjwvel}, that
\begin{align*}
\begin{aligned}
&\nsnonlinb{\partial_t\wvel}{\wvel},\,
\nsnonlinb{\wvel}{\partial_t\wvel},\,
\nsnonlinb{\partial_t\wvel}{\vvel},\,
\nsnonlinb{\vvel}{\partial_t\wvel}\in\LR{q}(\grp)\cap\LR{r}(\grp).
\end{aligned}
\end{align*}
Consequently, we have
\begin{align*}
\partial_t\bb{\partial_t\wvel} -\Delta\bb{\partial_t\wvel} -\rey\partial_1\bb{\partial_t\wvel} 
\in\LRsigmacompl{q}(\grp)\cap\LRsigmacompl{r}(\grp).
\end{align*} 
Combining again Lemma \ref{lt_TPOseenMappingThmLem} and Lemma \ref{lt_TPOseenUniqueness}, we conclude the improved regularity
\begin{align}\label{ss_RegularityThmPartailtRegularity}
\partial_t\wvel\in\WSRsigmacompl{2,1}{q}(\grp)\cap \WSRsigmacompl{2,1}{r}(\grp)
\end{align}
of the time derivative of $\wvel$. We can establish the same improved regularity of spatial derivatives of $\wvel$.
For this purpose we simply observe that 
\begin{align*}
\begin{aligned}
&\nsnonlinb{\partial_j\wvel}{\wvel},\,
\nsnonlinb{\wvel}{\partial_j\wvel},\,
\nsnonlinb{\partial_j\wvel}{\vvel},\,
\nsnonlinb{\wvel}{\partial_j\vvel},\,
\nsnonlinb{\partial_j\vvel}{\wvel},\,
\nsnonlinb{\vvel}{\partial_j\wvel}\in\LR{q}(\grp)\cap\LR{r}(\grp),
\end{aligned}
\end{align*}
which implies, by applying $\partial_j$ on both sides in \eqref{ss_RegularityThmwvelinfty}, that
\begin{align*}
\partial_t\bb{\partial_j\wvel} -\Delta\bb{\partial_j\wvel} -\rey\partial_1\bb{\partial_j\wvel} 
\in\LRsigmacompl{q}(\grp)\cap\LRsigmacompl{r}(\grp).
\end{align*} 
Employing yet again Lemma \ref{lt_TPOseenMappingThmLem} and Lemma \ref{lt_TPOseenUniqueness}, we obtain
\begin{align}\label{ss_RegularityThmPartailjRegularitywvel}
\grad\wvel\in\WSRsigmacompl{2,1}{q}(\grp)\cap \WSRsigmacompl{2,1}{r}(\grp).
\end{align}
We now turn our attention to $\vvel$. Applying $\proj$ to both sides in \eqref{ss_nsongrp}, we deduce 
\begin{align}\label{ss_RegularityThmEqForvvel}
\begin{pdeq}
&-\Delta\vvel -\rey\partial_1\vvel = \hproj\proj f - \hproj\Bb{\proj\bb{\nsnonlin{\wvel}} +\nsnonlin{\vvel} }  && \tin\R^3,\\
&\Div\vvel =0 && \tin\R^3.
\end{pdeq}
\end{align} 
Recalling 
\eqref{ss_StrongSolThmEmbeddingLemGradvinfty}, 
\eqref{ss_StrongSolThmEmbeddingLemvinfty}, 
\eqref{ss_StrongSolThmEmbeddingLemwinfty} and
\eqref{ss_RegularityThmwvelinfty}, 
one readily verifies
\begin{align*}
\begin{aligned}
&\proj\nb{\nsnonlinb{\partial_j\wvel}{\wvel}},\,
\proj\nb{\nsnonlinb{\wvel}{\partial_j\wvel}},\,
\nsnonlinb{\partial_j\vvel}{\vvel},\,
\nsnonlinb{\vvel}{\partial_j\vvel}\in\LR{q}(\R^3)\cap\LR{r}(\R^3).
\end{aligned}
\end{align*}
Thus, applying $\partial_j$ on both sides in \eqref{ss_RegularityThmEqForvvel} we obtain
\begin{align*}
-\Delta\bb{\partial_j\vvel} -\rey\partial_1\bb{\partial_j\vvel} \in \LRsigma{q}(\R^3)\cap\LRsigma{r}(\R^3).
\end{align*} 
By Lemma \ref{lt_OseenMappingThmLem} and Lemma \ref{lt_OseenUniquenessLem}, we conclude that
\begin{align}\label{ss_RegularityThmPartailjRegularityvvel}
\grad\vvel\in\xoseen{q,r}(\R^3).
\end{align}
Summarizing \eqref{ss_RegularityThmPartailtRegularity}, \eqref{ss_RegularityThmPartailjRegularitywvel}, and \eqref{ss_RegularityThmPartailjRegularityvvel},
we have established similar regularity for the first order derivatives of $\wvel$ and $\vvel$ as we had originally for $\wvel$ and $\vvel$. 
More precisely, we have
\begin{align*}
\forall\,\snorm{\alpha}\leq 1,\ \snorm{\beta}+\snorm{\kappa}\leq 1:\quad 
(\partial_x^\alpha\vvel,\partial_x^\beta\partial_t^\kappa\wvel)\in \xoseen{q,r}(\R^3)\times\WSRsigmacompl{2,1}{q,r}(\grp).
\end{align*}
Iterating the argument above with $(\partial_x^\alpha\vvel,\partial_x^\beta\partial_t^\kappa\wvel)$ in the role of 
$(\vvel,\wvel)$, we obtain the same regularity for all higher order derivatives as well, that is, 
\begin{align}\label{ss_RegularityThmFinalvwvelregularity}
&\forall\snorm{\alpha}\leq m,\ \snorm{\beta}+\snorm{\kappa}\leq m:
\quad 
(\partial_x^\alpha\vvel,\partial_x^\beta\partial_t^\kappa\wvel)\in \xoseen{q,r}(\R^3)\times\WSRsigmacompl{2,1}{q,r}(\grp).
\end{align}
Concerning the pressure term $\upres$, we clearly have
\begin{align}\label{ss_RegularityThmPressureEq}
\grad\upres = \bp{\id-\hproj}\bb{f-\nsnonlin{\uvel}}.
\end{align}
From \eqref{ss_RegularityThmFinalvwvelregularity} one easily deduces 
\begin{align*}
\forall\snorm{\beta}+\snorm{\kappa}\leq m:\quad \partial_x^\beta\partial_t^\kappa\bb{\nsnonlin{\uvel}} \in\LR{q}(\grp)^3\cap\LR{r}(\grp)^3.
\end{align*}
Taking derivatives in \eqref{ss_RegularityThmPressureEq} and recalling Lemma \ref{ss_PressureMappingLem}, we thus obtain
\begin{align*}
\forall\snorm{\beta}+\snorm{\kappa}\leq m:\quad \partial_x^\beta\partial_t^\kappa\upres\in\xpres{q,r}(\grp),
\end{align*}
which concludes the theorem.
\end{proof}

We have now shown the main results for the reformulated version \eqref{ss_nsongrp} of the system \eqref{intro_nspastbodywholespace}--\eqref{intro_timeperiodicsolution} in 
the setting of the group $\grp$. It remains to verify that the results carry over to the original time-periodic setting in $\R^3\times\R$. For this 
purpose, we make two basic observations. 

\begin{lem}\label{quotientmapLiftingProps}
Let $k\in\N_0$ and $q\in(1,\infty)$.
The quotient mapping $\quotientmap:\R^3\times\R\ra\grp$ induces, by lifting, an embedding $\WSR{k}{q}(\grp)\embeds\WSRloc{k}{q}\bp{\R^3\times\R}$ with 
\begin{align}\label{quotientmapLiftingProps_ExtensionOfGrpDerivatives}
\forall\snorm{(\alpha,\beta)}\leq k :\quad (\partial_t^\beta\partial_x^\alpha\uvel)\circ\quotientmap = {\partial_t^\beta\partial_x^\alpha (\uvel\circ\quotientmap)}.
\end{align}
Similarly, lifting by $\quotientmap$ induces the embedding $\WSRsigma{2,1}{q}(\grp)\embeds\WSRloc{2,1}{q}\np{\R^3\times\R}^3$, with the relevant derivatives satisfying \eqref{quotientmapLiftingProps_ExtensionOfGrpDerivatives}, and for $r\geq q$ also
$\xpres{q,r}(\grp)\embeds\WSRloc{1,0}{r}\np{\R^3\times\R}$.
\end{lem}
\begin{proof}
Consider $\uvel\in\WSR{k}{q}(\grp)$ and let $\phi\in\CRci(\R^3\times\R)$. Let $\set{\psi_k}_{k\in\Z}\subset\CRci(\R)$ be a partition of unity subordinate to
the open cover $\setc{\bp{\frac{k}{2}\per,\np{\frac{k}{2}+1}\per}}{k\in\Z}$ of $\R$. The $\per$-periodicity of $\uvel\circ\quotientmap$ then implies
\begin{align*}
\int_{\R^3}\int_{\R} \uvel\circ\quotientmap\cdot\partial_t^\beta\partial_x^\alpha\phi\,\dt\dx
&=\int_{\R^3}\int_{\R} \uvel\circ\quotientmap\cdot\partial_t^\beta\partial_x^\alpha\Bb{\sum_{k\in\Z}\psi_k\phi}\,\dt\dx\\
&=\sum_{k\in\Z}\,\int_{\R^3}\int_0^\per \uvel\circ\bijection\cdot\partial_t^\beta\partial_x^\alpha\Bb{(\psi_k\phi)(x,t+\frac{k}{2}\per)}\,\dt\dx\\
&=\sum_{k\in\Z}\,\int_{\grp} \uvel\cdot\partial_t^\beta\partial_x^\alpha\Bb{(\psi_k\phi)(\cdot,\cdot+\frac{k}{2}\per)\circ\bijectioninv}\,\dg\\
&=(-1)^{\snorm{(\alpha,\beta)}}\sum_{k\in\Z}\,\int_{\grp} \partial_t^\beta\partial_x^\alpha\uvel\cdot\Bb{(\psi_k\phi)(\cdot,\cdot+\frac{k}{2}\per)\circ\bijectioninv}\,\dg\\
&=(-1)^{\snorm{(\alpha,\beta)}}\int_{\R^3}\int_{\R} \bp{\partial_t^\beta\partial_x^\alpha\uvel}\circ\quotientmap\cdot\phi\,\dt\dx,
\end{align*}
from which we deduce \eqref{quotientmapLiftingProps_ExtensionOfGrpDerivatives}. The other statements follow analogously.
\end{proof}

\begin{lem}\label{bijectionHomeomorphism}
The mapping $\Pi=\pi_{|\R^3\times(0,\per)}$ induces, by lifting, a homeomorphism between 
$\WSR{k}{q}(\grp)$ and $\WSRper{k}{q}\bp{\R^3\times(0,\per)}$,  
$\WSRsigmacompl{2,1}{q,r}(\grp)$ and $\WSRsigmapercompl{2,1}{q,r}\bp{\R^3\times(0,\per)}$, as well as between
$\xpres{q,r}(\grp)$ and $\xpres{q,r}\bp{\R^3\times(0,\per)}$.
\end{lem}
\begin{proof}
The spaces $\CRci(\grp)$ and $\CRciper\bp{\R^3\times[0,\per]}$ are dense in the Sobolev spaces $\WSR{k}{q}(\grp)$ and $\WSRper{k}{q}\bp{\R^3\times(0,\per)}$, respectively. 
By construction of the differentiable structure on $\grp$,
lifting by $\bijection$ is a homeomorphism between $\bp{\CRci(\grp),\norm{\cdot}_{k,q}}$ and the space $\bp{\CRciper\bp{\R^3\times[0,\per]},\norm{\cdot}_{k,q}}$.
It follows that this mapping extends to a homeomorphism
between $\WSR{k}{q}(\grp)$ and $\WSRper{k}{q}\bp{\R^3\times(0,\per)}$. The other statements follow analogously.
\end{proof}

\begin{proof}[Proof of Theorem \ref{ExistenceAndUniquenessThm}]
It is easy to verify that lifting by $\bijection$ is a homeomorphism between $\LR{q}(\grp)$ and $\LR{q}\bp{\R^3\times(0,\per)}$ with 
$\norm{f\circ\bijectioninv}_{q,\grp}=\per^{-\frac{1}{q}}\norm{f}_{q,\R^3\times(0,\per)}$.
We thus choose $\const{ExistenceAndUniquenessThmConst}\leq \bp{\per^{-\frac{1}{q}}+\per^{-\frac{1}{r}}}^{-1}\const{ss_StrongSolThmEps}$, where 
$\const{ss_StrongSolThmEps}$ is the constant from Theorem \ref{ss_StrongSolThm}.
Consider now a vector field $f\in\LR{q}\bp{\R^3\times(0,\per)}^3\cap\LR{r}\bp{\R^3\times(0,\per)}^3$ satisfying \eqref{intro_timeperiodicdata} and \eqref{ExistenceAndUniquenessThmDataCond}. 
Then $\tf:=f\circ\bijectioninv$ satisfies \eqref{ss_StrongSolThmDataCond}, whence there exists, by Theorem \ref{ss_StrongSolThm}, a 
solution $(\tuvel,\tupres)\in\LRloc{1}(\grp)^3\times\LRloc{1}(\grp)$ to \eqref{ss_nsongrp} in the class \eqref{ss_StrongSolThmSolSpace} (with $\tuvel=\tvvel+\twvel$).
Letting $\uvel:=\tuvel\circ\quotientmap$ and $\upres:=\tupres\circ\quotientmap$, we deduce from Lemma \ref{quotientmapLiftingProps} and Lemma \ref{bijectionHomeomorphism}
that $(\uvel,\upres)$ is a solution to \eqref{intro_nspastbodywholespace}--\eqref{intro_timeperiodicsolution} in the class \eqref{ExistenceAndUniquenessThmSolSpace}.
By Lemma \ref{quotientmapLiftingProps} we further see that a vector field $\weakuvel$ is a weak solution in the sense of Definition \ref{UniquenessClassDef} corresponding to 
$f$ if and only if $\tweakuvel:=\weakuvel\circ\bijectioninv$ is a weak solution in the sense of Definition \ref{ss_UniquenessClassDef} corresponding to $\tf$.
Consequently, uniqueness of $\uvel$ in the class of physically reasonable weak solutions follows from Theorem \ref{ss_UniquenessThm}. Finally, we obtain directly
from Theorem \ref{ss_EnergyEqThm} that $\uvel$ satisfies the energy equality \eqref{EnergyEqEE}.
\end{proof}

\begin{proof}[Proof of Theorem \ref{RegularityThm}]
Follows directly from Theorem \ref{ss_RegularityThm} in combination with Lemma \ref{quotientmapLiftingProps} and Lemma \ref{bijectionHomeomorphism}.
\end{proof}

\begin{proof}[Proof of Corollary \ref{RegularitySmoothnessCor}]
The corollary follows from Theorem \ref{RegularityThm} by a standard localization argument combined with 
classical Sobolev embedding.
\end{proof}

%%%%%%%%%%%%%%%%%%%%%%%%%%%%%%%%%%%%%%%%%%%%%%%%%%%%%%%%%%%%%%
%%          Bibliography                                    %%
%%%%%%%%%%%%%%%%%%%%%%%%%%%%%%%%%%%%%%%%%%%%%%%%%%%%%%%%%%%%%%
\bibliographystyle{abbrv}

\end{document}